\documentclass[11pt, reqno]{amsart}
\usepackage{amsmath, amsthm, amscd, amsfonts, amssymb, graphicx, color, stmaryrd}
\usepackage[all,cmtip]{xy}	% diagrams
\usepackage{tikz-cd}
\usepackage[bookmarksnumbered, colorlinks, plainpages]{hyperref}
\usepackage{slashed}

\def\presuper#1#2%
  {\mathop{}%
   \mathopen{\vphantom{#2}}^{#1}%
   \kern-\scriptspace%
   #2}

\tikzset{%
    symbol/.style={%
        ,draw=none
        ,every to/.append style={%
            edge node={node [sloped, allow upside down, auto=false]{$#1$}}}
    }
}

\newcommand{\Rr}{\mathbb{R}}
\newcommand{\Cc}{\mathbb{C}}
\newcommand{\Zz}{\mathbb{Z}}

\newcommand{\B}{\mathcal{B}}
	
\newcommand{\D}{\mathcal{D}}	% diff. operators
\newcommand{\DV}{\presuper{\V}{\mathcal{D}}}
\newcommand{\Dd}{\mathbb{D}}	% equ. family of operators
\newcommand{\Dirac}{\slashed D}	% geo. Dirac operator
\newcommand{\Ee}{\mathbb{E}}	% equiv. bundle
\newcommand{\Exp}{\mathrm{Exp}}	
\newcommand{\F}{\mathcal{F}}
\renewcommand{\H}{\mathcal{H}}	% Hilbert-space
\renewcommand{\L}{\mathcal{L}}	% bd lin Operators
\newcommand{\C}{\mathbf{C}}	% ...
\newcommand{\Ct}{\widetilde{\C}}
\newcommand{\Cl}{\mathrm{Cl}}	% Clifford-Algebra
\newcommand{\A}{\mathcal{A}} 	% Algebra bising. Pseudos
\newcommand{\U}{\mathcal{U}}  	% comparison algebra
	
\newcommand{\Aroof}{\hat{A}(\nabla)} % \V-roof genus
\newcommand{\cliff}{\mathbf{c}}	% Cl. action
\newcommand{\n}{\underline{n}}	% normal vector
\renewcommand{\S}{\Symb}
\newcommand{\id}{\mathrm{id}}
\newcommand{\indV}{\mathrm{\presuper{\V}{\ind}}}  % KV trace
\newcommand{\indW}{\mathrm{\presuper{\W}{\ind}}} % bdy..
\newcommand{\OmegaV}{\Omega^1(\A)}
\newcommand{\Tr}{\mathrm{Tr}}			% (regularized) trace
\newcommand{\TrV}{\mathrm{\presuper{\V}{\Tr}}}	% \V-trace
\newcommand{\TrsV}{\mathrm{\presuper{\V}{\Tr_s}}} % super \V-trace
\newcommand{\TrW}{\mathrm{\presuper{\W}{\Tr}}}	% \V-trace
\newcommand{\etaV}{\mathrm{\presuper{\V}{\eta}}} % \V-eta function
\newcommand{\etaVc}{\mathrm{\presuper{\V_1}{\eta}}} % \V-eta function
\newcommand{\etaW}{\mathrm{\presuper{\W}{\eta}}} % \V-eta function
\newcommand{\tr}{\mathrm{tr}}			% trace
\newcommand{\LC}{\mathcal{\mathbf{LC}}}	% Lie category
\newcommand{\LCt}{\widetilde{\LC}}
\newcommand{\LG}{\mathcal{\mathbf{LG}}}
\newcommand{\Cstar}{\mathbf{C^{\ast}}} % C(T)-alg
\renewcommand{\H}{\mathcal{H}} % Hilbert space
\newcommand{\G}{\mathcal{G}}	% groupoid
\newcommand{\Gad}{\G^{ad}}	% adiabatic grpd
\newcommand{\Abdy}{\widetilde{\A}}	% boundary algebroid
\newcommand{\Gbdy}{\widetilde{\G}^{ad}} % semi-groupoid
\newcommand{\Had}{\H^{ad}}
\newcommand{\Gop}{\mathrm{\G^{(0)}}} % objects
\newcommand{\Hop}{\mathrm{\H^{(0)}}}

\newcommand{\Mor}{\mathrm{Mor}}		% arrows
\renewcommand{\O}{\mathcal{O}}

\renewcommand{\L}{\mathcal{L}} % bd operators
\newcommand{\R}{\mathcal{R}}
\renewcommand{\S}{\mathcal{\mathbf{S}}}	% Schwartz functions
	% ...

\newcommand{\V}{\mathcal{V}}		% vector fields
\newcommand{\W}{\mathcal{W}}		% ...
\newcommand{\N}{\mathcal{N}}		% normal bundle
\newcommand{\End}{\mathrm{End}}		% endos
\newcommand{\Hom}{\mathrm{Hom}}		% Hom bdl
\renewcommand{\P}{\mathcal{P}}		% projs
\newcommand{\Diff}{\mathrm{Diff}}	% diffeos
\newcommand{\supp}{\mathrm{supp}}	% supp
\newcommand{\fop}{f^{(0)}}		% obj morph
\newcommand{\ind}{\operatorname{ind}}
\newcommand{\iso}{\xrightarrow{\sim}}	% iso arrow
\newcommand{\nablaslash}{\nabla\hspace{-.8em}/\hspace{.3em}}
\newcommand*{\avint}{\mathop{\ooalign{$\int$\cr$-$}}}

\newcommand*{\avintY}{\mathop{\ooalign{$\int_Y$\cr$-$}}}
\newcommand{\avintV}{\mathrm{\presuper{\V}{\avint}}} % Melrose reg. + reg. 
\newcommand{\avintWY}{\mathrm{\presuper{\W}{\avintY}}} % Melrose reg. + reg. 
\newcommand*{\davintM}{\mathop{\ooalign{$\displaystyle \int_M$\cr$-$}}}
\newcommand*{\davintMc}{\mathop{\ooalign{$\displaystyle \int_{M_1}$\cr$-$}}}
\newcommand*{\davintY}{\mathop{\ooalign{$\displaystyle \int_Y$\cr$-$}}}
\newcommand{\davintVM}{\mathrm{\presuper{\V}{\davintM}}} % Melrose reg. + reg. 
\newcommand{\davintWY}{\mathrm{\presuper{\W}{\davintY}}} % Melrose reg. + reg. 
\newcommand{\davintVMc}{\mathrm{\presuper{\V_1}{\davintMc}}} % Melrose reg. + reg. 

\newtheorem{Thm}{Theorem}[section]
\newtheorem{Lem}[Thm]{Lemma}
\newtheorem{Prop}[Thm]{Proposition}

\theoremstyle{definition}
\newtheorem{Def}[Thm]{Definition}
\newtheorem{Exa}[Thm]{Example}

\newtheorem{Rem}[Thm]{Remark}

\begin{document}
\setcounter{page}{1}

%-------------------------- Pleased do not change the following line-------------------------------------------
%\noindent \textcolor[rgb]{0.99,0.00,0.00}{This is a submission to one of journals of TMRG: BJMA/AFA}\\[.5in]
%--------------------------------------------------------------------------------------------------------------

\title{Index formulae for mixed boundary conditions on manifolds with corners}

\author[Karsten Bohlen]{Karsten Bohlen}

\address{$^{1}$ Universit\"at Regensburg, Germany}
\email{\textcolor[rgb]{0.00,0.00,0.84}{karsten.bohlen@mathematik.uni-regensburg.de}}

%\dedicatory{This paper is dedicated to Professor ABCD}

\subjclass[2000]{Primary 19K56; Secondary 46L80.}

\keywords{adiabatic groupoid, boundary conditions, Lie manifold.}

% \date{Received: xxxxxx; Revised: yyyyyy; Accepted: zzzzzz.}

\begin{abstract}
We investigate the problem of calculating the Fredholm index of a geometric Dirac operator subject to local (e.g. Dirichlet and Neumann) and non-local (APS) boundary conditions posed on the strata of a manifold with corners.
The boundary strata of the manifold with corners can intersect in higher codimension. To calculate the index we introduce a glueing construction and a corresponding Lie groupoid. We describe the Dirac operator
subject to mixed boundary conditions via an equivariant family of Dirac operators on the fibers of the Lie groupoid. Using a heat kernel method with rescaling we derive a general index formula of the Atiyah-Singer type. 
\end{abstract} \maketitle

\section{Introduction}

% introduce new invariant: renormalized analytic index (may take real (non-integer) values, in particular may not equal Fredholm index)
On manifolds with boundary the Atiyah-Patodi-Singer (APS) index theorem yields an index formula for Fredholm operators which consists of a local contribution, depending on the Riemannian metric and a non-local contribution, the $\eta$-invariant.
The boundary conditions considered in the APS-theory are global projection conditions. In the present work our goal is to describe a generalization of the APS-theory which involves general and mixed boundary conditions.
Historically, boundary value problems were studied on manifolds with boundary, where the boundary may consist
of disjoint components, on which different types of boundary conditions could be imposed.
In the classical setup the problems under consideration involve questions of well-posedness for solutions, Fredholm conditions
for the operators involved, as well as the index theory of partial differential equations subject to mixed boundary conditions.
A mixed boundary value problem by definition is a partial differential equation subject to different boundary
conditions on the different pieces of the boundary. In this note our main focus lies on the index theory for mixed boundary value problems and for convenience we will
restrict attention to Dirac type operators. The classical Dirichlet problem and Neumann problem in such a mixed setup was investigated by Zaremba in 1910, \cite{Zar}. 
 Index formulae have been obtained by Gromov and Lawson \cite{GL} (relative index theorem) and Freed \cite{F} for 
Dirac type operators on manifolds of odd dimension, see also \cite{BB}.
In these works the boundary is assumed to be compact and to possibly consist of disjoint pieces.  
We think it is an interesting question to study more general domains, which may consist of boundaries
which are only piecewise smooth. Examples considered in the literature are Lipschitz domains \cite{JK}, \cite{MMT}, \cite{MT} and 
the complex analytic analogues of pseudoconvex domains \cite{Fe}. 
A convenient framework for the consideration of such singular domains are in terms of manifolds with corners and singular
structures defined on manifolds with corners, cf. \cite{ALN}, \cite{M}.
In our setup we study compact manifolds with corners which by definition consist of a boundary stratified by immersed
submanifolds. Therefore, in contrast to previous results, we in particular allow intersecting boundary components and
examine Dirac type operators subject to mixed boundary conditions on boundary components.
The main focus of this work however will be the index theory of such boundary value problems. 
We obtain index formulae in the spirit of the index theorem of Atiyah-Patodi-Singer.
In the case without mixed boundary conditions, posed on so-called regular strata, our results yield the known 
Atiyah-Patodi Singer index formula on a manifold with boundary \cite{M} and more generally on manifolds with corners 
\cite{BS}. 

\subsection{Overview}

\subsection*{Dirac operators}

A Lie manifold is a triple $(M, \A, \V)$, where $M$ is a compact manifold with corners and
$\V \subset \Gamma(TM)$ is a Lie algebra of smooth vector fields, cf. \cite{ALN}. 
Moreover, $\V$ is assumed to be a subalgebra of the Lie algebra $\V_b$ of all vector fields tangent to the boundary strata  
and a finitely generated  projective $C^{\infty}(M)$-module.
The compact manifold with corners $M$ is thought of as a compactification of a non-compact manifold with a degenerate,
singular metric which is of product type at infinity. 
We denote by $\partial M$ the (stratified) boundary of $M$ and by  $M_0 = M \setminus \partial M$ the interior.
By the Serre-Swan theorem there exists a vector bundle $\A \to M$ such that $\Gamma(\A) \cong \V$.
We make the standing assumption that $\A_{M_0} \cong TM_0$, the tangent bundle on the interior. 
The bundle $\A$ has the structure of a Lie algebroid and we fix throughout $\varrho \colon \A \to TM$ to denote the \emph{anchor} of the Lie algebroid, see e.g. \cite{Mack}. Moreover, we need a Lie groupoid $\G \rightrightarrows M$.
It is known that for any Lie structure with our assumption there is an $s$-connected Lie groupoid $\G$ such that $\A(\G) \cong \A$, cf. \cite{D}. 
Since we rely on the heat kernel of an arbitrary integrating Lie groupoid, we make the standing assumption that all Lie manifolds considered in this work have
heat kernels which are (transversally) smooth. It is conjectured, but yet not resolved in full generality, whether every Lie manifold has an integrating Lie groupoid with a smooth heat kernel, cf. \cite{So}, \cite{BS}.  
An \emph{$\A$-metric} is a Riemanian metric $g = g_{\A}$ on the interior $M_0$ which extends to a positive definite symmetric bilinear form
on $\A$, i.e. a euclidian structure on $\A$. We call $\nabla$ an \emph{$\A$-connection} if $\nabla$ is the Levi Civita connection on the interior $M_0$ which extends to a connection
defined on $\A$ as described in \cite{ALN2}. Likewise if $W$ is a $\Cl(\A)$-module we call $\nabla^W$ an \emph{admissible $\A$-connection} if
\[
\nabla_X^W(c(Y) \varphi) = c(\nabla_X Y) \varphi + c(Y)(\nabla_X^W \varphi), \ X, Y \in \Gamma(\A), \ \varphi \in \Gamma(W).
\]
Where Clifford multiplication is given by $c \colon \Cl(\A) \to \End(W)$ and $\nabla$ is the Levi-Civita $\A$-connection with respect to a given $\A$-metric $g$, see also \cite{LN}, \cite{ALN2}.
With an admissible $\A$-connection we can define a general Atiyah-Singer type geometric Dirac operator $D = D^W$, cf. \cite{LN}.
The \emph{vector representation} furnishes a corresponding $\G$-invariant geometric Dirac operator $\Dirac$ such that $\varrho(\Dirac) = D$, cf. \cite{LN}.
Here the vector representation is characterized by the equality: $(\varrho(P) f) \circ r = P(f \circ r)$, where $r$ is the range map of the groupoid (a surjective submersion),
$P \in \End(C^{\infty}(\G))$ and $f \in C^{\infty}(M)$, see also \cite{ALN}, \cite{NWX}.
Intuitively, we can think of the boundary strata of $M$ as being \emph{pushed to infinity}. This is the correct setup for APS-index theory as considered in \cite{M} for the special case of the maximal Lie structure $\V = \V_b$ and a manifold with boundary.

\subsection*{Local boundary conditions}

On the other hand we would like to take into account also \emph{local boundary conditions}, e.g. Dirichlet and Neumann condition in the setting of Lie manifolds.
To this end we need to modify slightly the definition of a Lie manifold and consider so-called Lie manifolds with boundary as introduced in \cite{AIN}. 
As before, a \emph{Lie manifold with boundary} consists of a manifold $M$ with corners, whose boundary
hyperfaces are $F_1, \dots, F_k$, i.e.\ $M=M_0 \cup F_1 \cup\cdots\cup F_k$. However, one boundary hyperface,
called the \emph{regular boundary}, say  $Y:= F_1$, has a special role, similar to a boundary component of a Riemannian manifold with boundary.
If we glue two copies of $M$ along $Y$, then we obtain the double $M \cup_{Y} M$. We do this doubling
such that the boundary of the double at the interior of $Y \cap F_i$, $i>1$ has no corner.
To keep the distinction in mind we use throughout the notation $\F_1(M)$ for the \emph{singular} boundary hyperfaces of $M$, as opposed to the regular boundary stratum $Y$.
Roughly speaking, a Lie structure with boundary on $M$ is defined as the restriction of a Lie structure on $M \cup_{Y} M$
to one of the copies of $M$, see \cite{AIN}.
for details. If $Y_0$ is the interior of $Y$, then the map $\varrho$ defines a bundle isomorphism from $\A|_{Y_0}$ to $TM|_{Y_0}$, and
$Y$ carries an induced Lie structure, denoted by $\B$. In the following
$\W:=\Gamma(\B)\subset \Gamma(TY)$. We want to consider geometric admissible Dirac operators with boundary conditions posed with regard to the hypersurface $Y$.
Here we also need to modify the definition of admissible Dirac operators for local boundary conditions. We do so in the main body of the paper.
In order to study general and mixed boundary conditions we modify also the definition of the Lie manifold with boundary further. We introduce the concept of a \emph{decomposed Lie manifold} which consists
of two parts $M = M_1 \cup_Y M_2$, glued at the regular hypersurface $Y$. 

\subsection{Organization of the paper}

In the second Section we recall the definition of a Lie manifold with boundary and study the properties of first order differential operators on Lie manifolds with boundary.
The third Section is concerned with decomposed Lie manifolds. We show that to any decomposed Lie manifolds there is a Lie groupoid integrating the Lie structure.
In the fourth section we define a Lie semi-groupoid and a convolution $C^{\ast}$-algebra for boundary value problems. The main result is 
the continuity of the field of $C^{\ast}$-algebras associated to the adiabatic semi-groupoid.
We show that there is a functional calculus taking values in the reduced $C^{\ast}$-algebra of the semi-groupoid.
In the fifth section we define the renormalized trace on Lie manifolds with boundary. Then we give a proof of the index theorem on
decomposed Lie manifolds with mixed boundary conditions which relies on a rescaling technique and the previously introduced functional calculus for the adiabatic semi-groupoid.
We also give the Fredholm conditions for Dirac operators and criteria for the equality of the renormalized index with the Fredholm index.

\section{Geometric Dirac operators}

\label{II}

% first order operators on measured Lie manifolds
% existence of tubular neighborhood theorem
% boundary symmetry
% normal form of operators
% spin Lie manifolds => oriented Lie manifolds => measured Lie manifolds

First order differential operators on a manifold with a Lie structure at infinity (a Lie manifold) are operators
which are contained in the enveloping algebra generated by the Lie structure.
For a Lie manifold with boundary we discuss differential operators of first order which are symmetric to the boundary, a notion
we introduce in this section. These operators are studied in the setting of measured Lie manifolds, a slightly more general
context than oriented Lie manifolds. An essential feature of Lie manifolds with boundary is the existence of a tubular neighborhood.
Since the interior of any Lie manifold is in particular a manifold with bounded geometry by \cite{ALN2}, we can contrast this to the case of a 
manifold with bounded geometry and boundary. In the latter case the existence of a tubular neighborhood is part of the definition of 
a manifold with bounded geometry and boundary.
For the (slightly) more restrictive class of Lie manifolds with boundary the existence of a tubular neighborhood can be derived as a theorem, i.e. as a consequence of the definition.
This is an advantage of the category of Lie manifolds over other, slightly more general categories of non-compact manifolds. 
Later in this section we introduce geometric Dirac operators on spin Lie manifolds. It is for the latter class of Dirac
operators that we will pose additional boundary conditions and derive suitable index theorems.

\subsection*{Lie manifold with boundary}

We recall the definition of a Lie manifold with boundary from \cite{AIN}. 

% Lie mf with boundary (recall def'n)
\begin{Def}
A \emph{Lie manifold with boundary} is a Lie manifold $(M, \A, \V)$ together with a submanifold $(Y, \B, \W)$ such that the following
conditions hold:

\emph{1)} $(Y, \B, \W) \hookrightarrow (M, \A, \V)$ is a \emph{Lie submanifold}, i.e. $Y \subset M$ is a submanifold with corners where
$\B \to Y$ is a $C^{\infty}$-vector bundle such that $\Gamma(\B) \cong \W$ and $\B \hookrightarrow \A_{|Y}$ is a Lie subalgebroid.

\emph{2)} The submanifold $Y$ is \emph{transverse} in $M$, i.e. $T_p M = \mathrm{span} \{\varrho(\A_p), T_p Y\}$ for each $p \in \partial Y = Y \cap \partial M$. 

\label{Def:Liebdy}
\end{Def}

\begin{Rem}
\emph{i)} Condition \emph{2)} is equivalent to $T_x M = T_x Y + T_x F$ for each $x \in F \cap Y$ and each closed codimension one face $F \in \F_1(M)$. 
Notice that the Lie structure of vector fields $\W$ is - by definition of a Lie submanifold - a subalgebra of $\V_{|Y}$, precisely $\W = \{V_{|Y} : V \in \V, \ V_{|Y} \ \text{tangent to} \ Y\}$. 

\emph{ii)} Given a Lie manifold with boundary as above. Consider the exponential map $\exp \colon \A \to M$ which is the natural extension from the interior $\exp \colon TM_0 \to M_0$. 
Setting $\N := \frac{\A_{|Y}}{\B}$ the $\A$-normal bundle, where $\Gamma(\B) \cong \W$ and $\B \hookrightarrow \A_{|Y}$ is in particular a sub vector bundle.
We obtain the exact sequence of vector bundles
\[
\xymatrix{
\B \ar@{>->}[r] & \A_{|Y} \ar@{->>}[r]^{q} & \N.
}
\]

This sequence splits and denote by $\eta \colon \N \to \A_{|Y}$ the splitting. By a choice of an $\A$-metric, we can
choose $\eta$ as an isomorphism $\eta \colon \N \iso \B^{\perp}, q \circ \eta = \id_{\N}$. This furnishes the decomposition 
$\A_{|Y} \cong \B \oplus \N$. 

\emph{iii)} An application of the achnor $\varrho$ of the Lie algebroid $\A$ yields an isomorphism
\[
\frac{\A_p}{\B_p} = \N_p \iso \frac{T_p M}{T_p Y} \cong N_p Y, \ p \in Y. 
\]

Hence, in particular, $\N_{|Y_0} \cong N Y_0$ is an isomorphism over $Y_0$. 

Fix an $\A$-metric $g = g_{\A}$ on $M$ and define the orthogonal complement $\B^{\perp}$ using $g$. 
We consider the induced exponential map, by observing that $\varrho_{\B^{\perp}}$ is injective onto its image, i.e. we have
\[
\xymatrix{
\varrho_{|\B^{\perp}} \colon \B^{\perp} \ar@{>->>}[d] \ar@{>->}[r] & \varrho(\B^{\perp}) \subset TM \ar[d] \\
\N \ar[r] & NY.
}
\]

For each $p \in F$, for a hyperface $F$ in $M$, we have $\varrho(\A_p) \subset T_p F$. We set $\exp^{\nu} := \exp_{|\B^{\perp}} \colon \B^{\perp} \to M$,
for the normal exponential map, which is a well-defined local diffeomorphism by the previous discussion. 
\label{Rem:Liebdy}
\end{Rem}

\subsection*{Measured manifolds}

% measured Lie mf's
For a given $\A$-metric $g$ on a Lie manifold $(M, \A, \V)$ we can associate a volume form which we denote throughout by $\mu_g$.

\begin{Def}
A Lie manifold $(M, \A, \V)$ is called \emph{measured} if there is a nowhere-vanishing smooth one-density $\mu$ on $M$ such that there is an $\A$-metric $g$ on $M$ with
$\mu = \mu_g$. 
\label{Def:measured}
\end{Def}

\begin{Prop}
Let $(M, \A, \V)$ be a Lie manifold with boundary $(Y, \B, \W)$ such that $(M, \mu)$ is measured. Then there is an induced
one-density $\nu$ on $Y$ obtained from an induced $\B$-metric $g_{\partial}$ on $Y$ which turns $(Y, \nu)$ into a measured Lie manifold. 
\label{Prop:measured}
\end{Prop}

\begin{proof}
Let $g$ be an $\A$-metric on $M$ such that $\mu = \mu_g$. Denote by $\B^{\perp}$ the complement of $\B$ defined with the metric $g$. 
If $g_{\partial}$ is a metric obtained by restriction to $Y$, then $g_{\partial}$ is a $\B$-metric on $Y$ and $\mu_{g_{\partial}} = \nu$ yields a one-density such that $\nu = \mu_{|Y}$.  
\end{proof}

\begin{Rem}
If $(M, \A, \V)$ is orientable, then we can trivialize the top degree part of $\Lambda^{\bullet} \A^{\ast}$ in order to obtain a global density
$\mu = \mu_g$. Then $\mu$ yields a measured manifold $(M, \mu)$. In particular any spin Lie manifold is orientable, therefore measured.
\label{measured}
\end{Rem}

\subsection*{Global tubular neighborhood}

% global tubular neighborhood theorem

% def'n cylinder sets

We introduce the notation $Y_{I} := I \times Y$ for an interval $I \subset \Rr$. Also write $Y_{(\epsilon)} := (-\epsilon, \epsilon) \times Y$. 
By a \emph{tubular neighborhood} of $Y$ in $M$ we mean a local diffeomorphism $\Psi \colon M \to M$ with an open neighborhood $Y \subset \U \subset M$ such that for a suitable $\epsilon > 0$ the restriction of $\Psi$ yields a 
diffeomorphism $\Psi \colon \U \iso Y_{(\epsilon)}$. 

\begin{Def}
Given a Lie manifold $(M, \A, \V)$ with boundary $(Y, \B, \W)$, a \emph{boundary defining function} $\rho_Y$ adapted to a tubular neighborhood $\Psi \colon \U \iso Y_{(\epsilon)}$ is an element of $C^{\infty}(M, \Rr)$ with
$\{\rho_Y = 0\} = Y$ and non-vanishing $d \rho_Y$ at $Y$ such that
\[
(\rho_Y \circ \Phi^{-1})(x', x_n) = x_n, \ x' \in Y, x_n \in (-\epsilon, \epsilon). 
\]
\label{Def:tubnbhd}
\end{Def}

\begin{Thm}
Let $(M, \A, \V)$ be a measured Lie manifold with boundary $(Y, \B, \W)$. There is an open neighborhood $Y \hookrightarrow \U \hookrightarrow M$
and a local diffeomorphism $\Phi \colon M \to M$ such that $\Phi \colon (u, \psi) \in \U \iso Y_{[0,r)}$ for some $r > 0$.
Additionally, $\Phi$ has the following properties: Denote by $\n$ a fixed normal vector field $\n \in \Gamma(\A_{|Y})$ and
by $\tau$ the accompaning $1$-form such that $\tau(\n) = 1$ and $\tau_{|Y} = 0$. 

\emph{i)} We have $Y = u^{-1}(0)$, where $u \equiv \rho_Y$ denotes the boundary defining function of $Y$. 

\emph{ii)} $\Phi_{|Y} = \id_Y$ and the following diagram commutes
\[
\xymatrix{
\U \ar[d] \ar@{>->>}[r]^{\Phi} & Y_{[0,r)} \ar[d]_{\mathrm{pr}_1} \\
Y \ar@{>->>}[r]^{\id = \Phi_{|Y}} & Y.
}
\]

\emph{iii)} $d \Phi(\n) = \partial_u$ along $Y$. 

\emph{iv)} $\tau(\n) = du$ along $Y$. 

\emph{v)} $\Phi_{\ast}(\mu) = |du| \otimes \nu$ on $Y_{[0,r)}$ where $\mu$ is the measure of $M$ and $\nu$ the induced measure on $Y$. 

\label{Thm:tubnbhd}
\end{Thm}

% proof: see AIN
\begin{proof}
Combine Proposition \ref{Prop:measured} with the tubular neighborhood theorem for Lie manifolds with boundary in \cite[Theorem 2.7]{AIN}.
\end{proof}

% consequence: integral formula & Green formula

\subsection*{Model operators}

% fix standard notation
We consider first order differential operators on a Lie manifold with boundary which are symmetric to the boundary. 
In the following we fix a Lie manifold $(M, \A, \V)$ with boundary $(Y, \B, \W)$ which is measured and we fix the volume element $\mu$. 
We also fix hermitian smooth vector bundles $E, F \to M$ and the normal vector field $\n$ with $1$-form $\tau$ as well as a tubular neighborhood $\Phi \colon \U \to Y_{[0,r)}$ for some $r > 0$ from Theorem \ref{Thm:tubnbhd}.

\begin{Def}
Let $D \in \Diff_{\V}^1(M, E, F)$, then $D$ is called \emph{boundary symmetric} (with regard to $\n$) if $D$ is elliptic and setting
$\sigma_0 := \sigma_D(\tau)$ we have that $\sigma_0(x)^{-1} \circ \sigma_D(\xi) \colon E_x \to E_x$, $\sigma_D(\xi) \circ \sigma_0(x)^{-1} \colon F_x \to F_x$ are skew Hermitian
for each $x \in Y, \ \xi \in \B_x^{\ast}$. 
\label{Def:bdysym}
\end{Def}

\begin{Prop}
\emph{i)} The operator $D$ is boundary symmetric if and only if the formal adjoint $D^{\ast}$ is boundary symmetric.

\emph{ii)} Boundary symmetry is independent of the choice of volume form $\mu$. 

\label{Prop:bdysym}
\end{Prop}

% i) is immediate
% cf. ALN2, for calculation of formal adjoint D^{\ast} (not needed here)
As a preparation for the analysis of the index problem for Dirac operators we extend next a result of \cite{BB} to the context of Lie manifolds with boundary.

\begin{Lem}
Let $D \in \Diff_{\V}^1(M; E, F)$ be elliptic and boundary symmetric. Fix the notation $\sigma_{(u,x)} \colon E_x \to F_x, \ (u,x) \in Y_{[0,r)}$,
$\sigma := \sigma_D(du), \ \sigma_{(0,x)} = \sigma_D(\tau(x)), \ x \in Y$.  
Then there are elliptic $\W$-differential operators 
$A \colon C_c^{\infty}(Y_0, E_{|Y}) \to C_c^{\infty}(Y_0, E_{|Y})$, $\tilde{A} \colon C_c^{\infty}(Y_0, F_{|Y}) \to C_c^{\infty}(Y_0, F_{|Y})$
such that 
\begin{align}
D &= \sigma_t(\partial_t + A + R_t), \label{n1} \\
D^{\ast} &= -\sigma_t^{\ast}(\partial_t + \tilde{A}_t + \tilde{R}_t), \label{n2}
\end{align}

where $R_t \colon C_c^{\infty}(Y_0, E_{|Y}) \to C_c^{\infty}(Y_0, E_{|Y})$, $\tilde{R}_t \colon C_c^{\infty}(Y_0, F_{|Y}) \to C_c^{\infty}(Y_0, F_{|Y})$
are operators contained in $\Diff_{\W}^l(Y; E)$ and $\Diff_{\W}^l(Y; F)$ respectively, where $l \leq 1, \ t \in [0,r)$.
Additionally, $R_t, \ \tilde{R}_t$ fufill the following estimates for $f \in C_c^{\infty}(Y_0, E_{|Y}), \ g \in C_c^{\infty}(Y_0, F_{|Y})$
\begin{align}
& \|R_t f\|_{L_{\W}^2(Y)} \leq C(t \|A f\|_{L_{\W}^2(Y)} + \|f\|_{L_{\W}^2(Y)}), \label{Ra} \\
& \|\tilde{R}_t g\|_{L_{\W}^2(Y)} \leq C(t \|\tilde{A} g\|_{L_{\W}^2(Y)} + \|g\|_{L_{\W}^2(Y)}). \label{Rb}
\end{align}
\label{Lem:bdysym}
\end{Lem}

\begin{proof}
With $x \in Y$ we can write $\B_x^{\ast} = \{\xi \in \A_x^{\ast} : \xi(\n) = 0\}$.
If there is an $A$ such that \eqref{n1}, \eqref{n2} hold then observe that $\sigma_A(\xi) = \sigma_0(x)^{-1} \circ \sigma_D(\xi)$
using that $\sigma_0(x)^{-1} \circ \sigma_D(\xi)$ is skew Hermitian for $x \in Y, \ \xi \in \B_x^{\ast}$. 
The task is to find $A$ formally selfadjoint with principal symbol $\sigma_A$. By definition $\sigma_A$ is composed of invertible symbols,
hence such an $A$ is elliptic over $Y_{[0,r)}$. Setting $D = \sigma_t(\partial_t + \D_t)$ where $\D_t \in \Diff_{\W}^1(Y; E, E)$
is a family of elliptic operators with smooth coefficients, $t \in [0,r)$. 
Set $R_t := \D_t - A$, note that $\sigma(\D_0) = \sigma(A)$ and $R_0$ is a zero order operator. Therefore by a Taylor expansion we have the estimate
\[
\|R_t f\|_{L_{\W}^2(Y)} \leq C'(t \|f\|_{H_{\W}^1(Y)} + \|f\|_{L_{\W}^2(Y)}). 
\]

Since $R_t \colon \H_{\W}^1(Y) \to \H_{\W}^0(Y) = L_{\W}^2(Y)$ is bounded and $A$ is elliptic of first order we obtain by \cite{AIN}
\[
\|f\|_1 \leq C(\|f\|_0 + \|A f\|_0)
\]

which gives \eqref{Ra}. Via $\sigma_{D^{\ast}}(\xi) = -\sigma_D(\xi)^{\ast}$ we obtain
\begin{align*}
\sigma_{\tilde{A}}(\xi) &= \sigma_{D^{\ast}}(\tau(x))^{-1} \sigma_{D^{\ast}}(\xi) 
= (\sigma_D(\tau(x))^{\ast})^{-1} \circ \sigma_D(\xi)^{\ast} \\
&= (\sigma_D(x)^{\ast})^{-1} \circ \sigma_D(\xi) 
= (\sigma_D(\xi) \circ \sigma_0(x)^{-1})^{\ast}. 
\end{align*}

\end{proof}

% rk: Green formula

\subsection*{Twisted Dirac operators}

Fix a spin Lie manifold $(M, \A, \V)$, an $\A$-metric $g$ and an $\A$-connection $\nabla^W$. 
We recall the definition of the geometric twisted Dirac operators on Lie manifolds, cf. \cite{ALN2}. 

Fix a Clifford module $W \in \Cl(\A)$ such that $W = W^{+} \oplus W^{-}$ is $\Zz_2$-graded compatible with the Clifford action: $\cliff(\Cl(\A)^{+}) W^{\pm} \subset W^{\pm}, \ \cliff(\Cl(\A)^{-}) W^{\pm} \subseteq W^{\mp}$. 

\begin{Def}
The geometric Dirac operator $D = D^W$ on the Lie manifold is defined as $D = c \circ (\id \otimes \sharp) \circ \nabla^W$, where $\sharp$ is the isomorphism
$\A \cong \A^{\ast}$ induced by the fixed compatible metric $g$, $\nabla^W$ denotes the $\A$-connection and $c$ the Clifford multiplication:
\[
\xymatrix{
\Gamma(W) \ar[r]^{\nabla^W} & \Gamma(W \otimes \A^{\ast}) \ar[r]^{\id \otimes \sharp} & \Gamma(W \otimes \A) \ar[r]^{c} & \Gamma(W).
}
\]
\label{Def:Dirac}
\end{Def}

Since $c$ is a $\V$-operator of order $0$ and $\nabla^{W}$ is a $\V$-operator of order $1$ we see that $D$ is in $\Diff_{\V}^1(M; W)$.
Additionally, $\sigma_1(D) \xi = i c(\xi) \in \End(W)$, hence invertible for $\xi \not= 0$, and $D$ is elliptic.\medskip
% special case: spin Dirac operator
% \hom(W) vs. \Hom(W)

% add later: def'n in terms of \Abdy \subset \A

We check that a geometric Dirac operator on a Lie manifold with boundary is boundary symmetric. 

\begin{Prop}
Let $(M, \A, \V)$ be a spin Lie manifold with boundary $(Y, \B, \W)$. Then $(M, \A, \V)$ is measured. 
A geometric Dirac operator $D = D^W$ acting on the $C^{\infty}$ vector bundles $E, F \to M$ for a given Clifford module $W$ is boundary symmetric with regard to $Y$ for a fixed choice
of normal vector field $\n \in \Gamma(\A_{|Y})$. 
\label{Prop:Dirac}
\end{Prop}

\begin{proof}
Let $g = g_{\A}$ be an $\A$-metric. The Clifford relations are
\begin{align*}
2 g_{\A}(\xi, \eta) \id_{E_x} &= \sigma_D(\xi)^{\ast} \sigma_D(\eta) + \sigma_D(\eta)^{\ast} \sigma_D(\xi), \ \xi, \eta \in \A_x^{\ast}, \\
2 g_{\A}(\xi, \eta) \id_{F_x} &= \sigma_D(\xi) \sigma_D(\eta)^{\ast} + \sigma_D(\xi)^{\ast} \sigma_D(\eta), \ \xi, \eta \in \A_x^{\ast}. 
\end{align*}

Set $\sigma_0(x) = \sigma_D(\tau(x))$ where $\tau$ denotes the $1$-form associated to the normal vector $\n$, i.e. $g_{\A}(\tau(x), \xi) = 0$. 
Extend $\xi$ to $\A_x$ by setting $\xi(\n) = 0$. Then the Clifford relations yield $\sigma_0(x)^{-1} \circ \sigma_D(\xi) \colon E_x \to E_x, \ \sigma_D(\xi) \circ \sigma_0(x)^{-1} \colon F_x \to F_x$ are skew-hermitian
for each $x \in Y, \ \xi \in \B_x^{\ast}$. 
\end{proof}

\section{A glueing construction}

\label{III}

We introduce so-called decomposed Lie manifolds. These Lie manifolds consist of two parts, one part is a so-called Lie manifold
of cylinder type and the complementary part is a standard Lie manifold in its own right. Furthermore, we show that for the corresponding Lie algebroid
on a decomposed Lie manifold, we can obtain an integrating groupoid via a glueing of two groupoids, one on the cylinder part
of the manifold and the other on the complement. We show that the resulting glued groupoid can be endowed with a smooth structure in a natural way.

\begin{Def}
A \emph{decomposed Lie manifold} $(M, \A, \V)$ with hypersurface $(Y, \B, \W)$ is a Lie manifold such that $M = M_1 \cup M_2$ where $M_1 \cap M_2 = Y$.
Additionally, $M_2$ is a Lie manifold of \emph{cylinder type}, i.e. $M_2$ is diffeomorphic to a global tubular neighborhood of $Y$ in $M$. 
\label{Def:decomposed}
\end{Def}

Any Lie manifold with boundary can be glued to a decomposed Lie manifold, up to a choice of tubular neighborhood, which follows by Theorem \ref{Thm:tubnbhd} and the following discussion.
Let $(M, \A, \V)$ be a decomposed Lie manifold such that $M_1 \cup M_2 = M, \ M_1 \cap M_2 = Y$. 
In the following we consider two groupoids: $\G_1 \rightrightarrows \mathring{M}_1$ where $\mathring{M}_1 := M_1 \setminus Y$ as well as
$\G_2 \rightrightarrows Y \times [-1,0]$. Here $M_2 \cong Y \times [-1,0]$ is the cylinder-type part of $M$. 
Let $\U$ be a global tubular neighborhood of $Y$ in $M$ such that $\U = \U_{+} \cup_Y \U_{-}$ is decomposed in the following sense. 
Fix a boundary defining function, i.e. smooth function $\rho \colon M \to \Rr$ such that $Y = \{\rho = 0\}$ and $d \rho$ is non-vanishing on $Y$.
We also consider the strata $Y_{+} := \{\rho = 1\}$ and $Y_{-} := \{\rho = -1\}$.
Let $\U$ be decomposed into the collars $\U_{+} \cong (0,1]_u \times Y_{+}$ and $\U_{-} \cong [-1,0)_u \times Y_{-}$. 
The groupoids $\G_1, \ \G_2$ are adapted to $(M, \A, \V)$ and $(Y, \B, \W)$. Let $\V_{\mathring{M}_1}$ be the Lie structure of $\mathring{M}_1$. 
Denote by $\G_1 = \G(\mathring{M}_1) \rightrightarrows \mathring{M}_1$ the Lie groupoid integrating the Lie structure $\V_{\mathring{M}_1}$. 
Let $\H \rightrightarrows Y$ be the Lie groupoid integrating the Lie structure $\W$. 
Define $\G_2 := \H \times ([-1,0)^2 \cup \{0\} \times \Rr_{+}) \rightrightarrows Y \times [-1,0]$. 
The groupoid structure of $\G_2$ is given by the pair groupoid structure on $[-1,0)^2$ and $(\Rr_{+}, \cdot)$ viewed as a multiplicative group. 

We are now in a position to state the following result.

\begin{Thm}
The groupoid $\G := \G_1 \cup \G_2 \rightrightarrows M$ has the $C^{\infty}$-structure of a Lie groupoid such that $\A(\G) \cong \A$, i.e. $\G$ integrates
the Lie structure of the decomposed Lie manifold $(M, \A, \V)$. 
\label{Thm:decomposed}
\end{Thm}

\begin{proof}
Define the auxiliary groupoid $\tilde{\H} := \H \times \Rr \rtimes \Rr_{+} \rightrightarrows Y \times \Rr$, where $\Rr \rtimes \Rr_{+} \rightrightarrows \Rr$
is the semi-direct product groupoid given by the multiplicative action of $(\Rr_{+}, \cdot)$ on $(\Rr, \cdot)$. 
It is immediate that $\tilde{\H}$ is a Lie groupoid. Set $\G = \G_1 \cup \G_2 \rightrightarrows M$, then we exhibit a Lie groupoid structure
on $\G$ with the help of $\tilde{\H}$, using a transport of structure argument. 
The definition of $\G$, together with the tubular neighbhorhood theorem for Lie manifolds then implies that $\A(\G) \cong \A$. 
We fix the collar neighborhoods as previously defined. We have $\G_2 = \tilde{\H}_{Y \times [-1,0]}^{Y \times [-1,0]}$. By transversality
of $Y$ in $M$ we have a canonical diffeomorphism $(\G_1)_{\U}^{\U} \cong \H \times (0,1)^2$. 
Denote this diffeomorphism by $\psi$. We construct the diffeomorphism
\[
\varphi \colon \tilde{\H}_{Y \times (-1,1)}^{Y \times (-1,1)} \iso \G
\]

Here we restrict $\G$ without loss of generality to $\U$ and denote the restriction by the same letter.
This will give the $C^{\infty}$-structure at $0$. Note that 
\[
\tilde{\H}_{Y \times (-1,1)}^{Y \times (-1,1)} = \H \times [(-1,0)^2 \cup \{0\} \times \Rr_{+} \cup (0,1)^2]. 
\]

Hence $\varphi$ is given by gluing the identity and the diffeomorphism $\psi$. A similar type of argument shows that the structural maps
of $\G$ are $C^{\infty}$ maps. Therefore $\G$ is a Lie groupoid.
\end{proof}

\section{Continuous field of $C^{\ast}$-algebras}

\label{IV}

In this section we answer the question of how to associate a $C^{\ast}$-algebra to a semi-groupoid in a functorial way.
To begin with we introduce suitable categories of fields of smooth groupoids and semi-groupoids as well as the category
of fields of $C^{\ast}$-algebras. Then we describe a contravariant functor from the category of fields of smooth groupoids to the
category of fields of $C^{\ast}$-algebras, or more generally, $C_0(T)$-algebras, where $T$ is a compact Hausdorff space. 
We contrast this functoriality with the case of the smaller Lie category of smooth semi-groupoids.
The next question is under what conditions a given field of semi-groupoids furnishes a continuous field of $C^{\ast}$-algebras.
Here we are interested in a particular type of deformation semi-groupoid over the cylinder part of a given decomposed Lie manifold.
We show that this semi-groupoid fulfills the necessary condition for functoriality with regard to a natural class of representations on Hilbert space.
Then, as a preparation for our study of the heat kernel in the final section, we introduce a functional calculus on the convolution $C^{\ast}$-algebra of 
the deformation semi-groupoid on a decomposed Lie manifold. 

\subsection*{Functoriality}

We recall the definition of $C_0(T)$-$C^{\ast}$-algebras where $T$ is a Hausdorff topological space. Then we establish what criteria are needed for the continuity of fields
of $C^{\ast}$-algebras as defined via Lie groupoids or Lie semi-groupoids. 
Recall that for a $C^{\ast}$-algebra $M(A)$ denotes the maximal unital $C^{\ast}$-algebra which contains $A$ as an essential ideal. Denote by $Z(B)$ the center of a given $C^{\ast}$-algebra $B$.

\begin{Def}
A $C_0(T)$-algebra is a tuple $(A, \theta)$ where $\theta \colon C_0(T) \to Z M(A)$ is a $\ast$-homomorphism such that $\theta(C_0(T)) A = A$.
\label{Def:field}
\end{Def}

Note that $a \in A$ can be identified with a family $a = (a_x)_{x \in T}$. Here $a_x \in A_x := A / C_x A, \ C_x := \{f \in C_0(T) : f(x) = 0\}$. 
The action of functions on $T$ is implemented by $\theta$ and we often abuse notation by writing $f \cdot a$ instead of $\theta(f) \cdot a$. 
We also write $A^T := \overline{\theta(C_0(T))} A$ and call $A$ \emph{non-degenerate} if $A = A^T$. 

\begin{Exa}
Consider $A = C_0(T)$ with $M(C_0(T)) = C_b(T), \ Z M(C_0(T)) = C_b(T)$. We will provide further examples induced by
the vast plethora of $C^{\ast}$-algebras associated to Lie groupoids.
\label{Exa:field}
\end{Exa}

Given two $C_0(T)$-algebras $A, B$ an arrow is given by a $\ast$-homomorphism $\psi \colon A \to B$ which is $C_0(T)$-linear, i.e.
$\psi(f \cdot a) = f \cdot \psi(a)$ for each $f \in C_0(T), \ a \in A$. 
Denote by $\C^{\ast}(T)$ the category with objects the $C_0(T)$-algebras and $C_0(T)$-linear $\ast$-homomorphisms as arrows between objects.

A particular case of $C_0(T)$-algebra is that given by a field of groupoids. 
\begin{Def}
A \emph{field} of Lie groupoids is a triple $(\G, T, p)$ where $\G$ is a Lie groupoid, $T$ is a $C^{\infty}$-manifold and $p \colon \G \to T$ is a submersion.
We denote by $p_0$ the restriction of $p$ to $\Gop$.
\label{Def:gfields}
\end{Def}

The \emph{category of $T$-Lie groupoids} $\LG(T)$ consists of objects the fields of Lie groupoids.
Let $(\G, T, p), \ (\H, T, \tilde{p})$ be $T$-Lie groupoids. Recall that a strict Lie groupoid morphism is a tuple 
$(f, \fop) \colon \G(T) \to \H(T)$ such that the diagram
\[
\begin{tikzcd}[every label/.append style={swap}]
\G \ar[d, shift left] \ar{d} \ar{r}{f} & \H \ar[d, shift left] \ar[d] \\
\Gop \ar{r}{\fop} & \Hop.
\end{tikzcd}
\]

commutes. 
An arrow in the category $\LG(T)$ is a tuple $(f, \fop)$ with $f$ and $\fop$ two $C^{\infty}$-maps such that the following diagram commutes
\[
\begin{tikzcd}[every label/.append style={swap}]
T & \\
\G \ar{u}{p} \ar[d, shift left] \ar{d} \ar{r}{f} & \H \ar{ul}{\tilde{p}} \ar[d, shift left] \ar[d] \\
\Gop \ar{d}{p_0} \ar{r}{\fop} & \Hop \ar[dl, "\tilde{p}_0", swap] \\
T & 
\end{tikzcd}
\]

We only mention briefly that the category $\C^{\ast}(T)$ has many useful stability properties, e.g. it is closed with regard to the formation of ideals,
quotients, direct sums, suspensions and it is $C_0(T)$-stable. The correct tensor product inside the category is the completed maximal tensor product.
Note that $\psi(A) \subset B^T$ for $\psi$ a possibly degenerate arrow and $\Mor_{C_0(T)}(A, B) = \Mor_{C_0(T)}(A, B^T)$. 

We want to study semi Lie groupoids, i.e. Lie groupoids where inverses do not always exist. This can be viewed as merely
the category of $C^{\infty}$-manifolds with additional data and structural maps as given in a Lie groupoid. 
We are mainly interested in the question of how to assign a $C_0(T)$-algebra to a given Lie groupoid or semi Lie groupoid.
To clarify these issues we will first prove a functoriality result for the association
\[
\C^{\ast} \colon \LG(T) \to \C^{\ast}(T). 
\]

In the functor $\C^{\ast} = \C_T^{\ast}$ we routinely suppress the dependency on $T$. For the following discussion we also refer to \cite{LR}, Section 5. 

\textbf{The map on objects:} Let $(\G, T, p)$ be a $T$-Lie groupoid. We describe a left $C_0(T)$-module structure on $C_c^{\infty}(\G)$. 
Define $(fa)(\gamma) = f(p(\gamma)) a(\gamma), \ f \in C_0(T), \ a \in C_c^{\infty}(\G)$. 
As a consequence $f(a \ast b) = (fa) \ast b = a \ast (fb)$ for each $f \in C_0(T), \ a,b \in C_c^{\infty}(\G)$. 
Also $C_0(T) C_c^{\infty}(\G) = C_c^{\infty}(\G)$ since if $f \in C_c(T)$ is chosen such that $f \equiv 1$ on $p(\supp a)$, then $a = fa \in C_0(T) C_c^{\infty}(\G)$. 

\textbf{Completion:} Endow $C_c^{\infty}(\G)$ with the inductive limit topology $\tau_{\to}$. We have the following condition  
\begin{align}
& \forall \ \pi \colon C_c^{\infty}(\G) \to \L(\H) \ \text{continuous representation} \notag \\ 
& \exists \ \varphi \colon C_0(T) \to \L(\H) \ \text{unique representation s.t.} \notag \\ 
& \pi(fa) = \varphi(f) \pi(a). \tag{$L$} \label{L}
\end{align}

As shown in Lemma 1.13. of \cite{Re} this condition holds for Lie groupoids. 
By an application of \eqref{L} for given $f \in C_0(T), \ a \in C_c^{\infty}(\G)$
\begin{align*}
& \|\pi(fa)\| \leq \|\varphi(f)\| \|\pi(a)\| \leq \|f\| \|a\| \\
& \Rightarrow \|fa\| \leq \|f\| \|a\|. 
\end{align*}

Therefore, for each $f \in C_0(T)$, the mapping $C_c^{\infty}(\G) \ni a \mapsto f a \in C_c^{\infty}(\G)$ extends continuously to $C^{\ast}(\G)$. 
Hence $C^{\ast}(\G)$ has a canonical $C_0(T)$-Banach module structure. It needs to be checked that $C_0(T) C^{\ast}(\G)$ is closed in $C^{\ast}(\G)$ by a separate argument, cf. \cite{LR}. 
Then non-degeneracy is clear by $C_0(T) C^{\ast}(\G) \supset C_0(T) C_c^{\infty}(\G) = C_c^{\infty}(\G)$. 
We have $f(a \ast b) = (fa) \ast b$ and $(f a)^{\ast} = f^{\ast} a^{\ast}$ for $a, b \in C^{\ast}(\G), \ f \in C_0(T)$, hence
$C^{\ast}(\G)$ is a $C_0(T)$-algebra. 

\textbf{The map on arrows:} Let $(f, \fop) \colon \G(T) \to \H(T)$ be a strict morphism of $T$-Lie groupoids. Define $\C^{\ast}(f, \fop)$, in short
$\C^{\ast}(f) \colon C^{\ast}(\H) \to C^{\ast}(\G)$ via the assignment
\begin{align*}
\C^{\ast}(f)(a) = a \circ f, \ a \in C^{\ast}(\H). 
\end{align*}

Let $g \in C_0(T)$, then $\C^{\ast}(g \cdot a) = (g \cdot a) \circ f$ where $g$ acts via multipliers on $C^{\ast}(\H)$. 
Denote by $(C^{\ast}(\G), \theta)$ and $(C^{\ast}(\H), \tilde{\theta})$ the corresponding actions. By definition
$\C^{\ast}(f)(g \cdot a) = \C^{\ast}(f)(\tilde{\theta}(g) \cdot a)$. It is then routine to verify that $\C^{\ast}(f)$ is
$C_0(T)$-linear: Take $g \in C_0(T)$ fixed
\begin{align*}
\C^{\ast}(f)(g \cdot a) &= \C^{\ast}(f)(\tilde{\theta}(g) \cdot a)(\gamma) \\
&= \C^{\ast}(f)((g \circ \tilde{p}) a)(\gamma) = ((g \circ \tilde{p}) \cdot a) f(\gamma) \\
&= (g \circ \tilde{p})(f(\gamma)) \cdot a(f(\gamma)) = g(\tilde{p}(f(\gamma))) \cdot a(f(\gamma)) \\
&= g(p(\gamma)) a(f(\gamma)) = \theta(g) \cdot \C^{\ast}(f)(a)(\gamma) = g \cdot \C^{\ast}(f)(a). 
\end{align*}

In the final line we used the commutativity 
\[
\xymatrix{
\G \ar[dr]_{p} \ar[r]^{f} & \H \ar[d]_{\tilde{p}} \\
& T 
}
\]

Altogether we obtain a \emph{contravariant} functor $\C^{\ast} \colon \LG(T) \to \C^{\ast}(T)$. 

What remains to be studied is the continuity of the resulting field of $C^{\ast}$-algebras. 
\begin{Def}
A $C_0(T)$-algebra $A$ is \emph{continuous} if $x \mapsto \|a_x\| \in [0,\infty)$ is continuous for each $x \in T$.  
\label{Def:cont}
\end{Def}

Up until now the functoriality makes sense for $\C^{\ast}$ as well as $\C_r^{\ast}$. The second part, namely $\C^{\ast}$ mapping
to \emph{continuous} $C_0(T)$-algebras, requires that we restrict to \emph{amenable} Lie groupoids. The reason is found in \cite[Theorem 5.5]{LR}
where upper semi-continuity of the field relies on the full $C^{\ast}$-algebra of the groupoid. We study next the general functoriality and continuity for Lie semi groupoids. There are two things to notice at the outset:
\emph{i)} For general functoriality of $T$ semi-groupoids the condition \eqref{L} is no longer true in general, but becomes an axiom. 
\emph{ii)} The proof of continuity of the field of $C^{\ast}$-algebras associated to a Lie groupoid crucially requires amenability.
This is not really available anymore on Lie semi-groupoids. Let us address both of these problems in what follows. Denote by $\LC$ the \emph{Lie category}, consisting of Lie semi-groupoids as objects
and smooth functors as arrows between objects. In the same vein we also introduce the category of $T$ Lie semi-groupoids $\LC(T)$. 
Denote by $\LCt(T)$ the subcategory of $\LC(T)$ consisting of objects the Lie semi-groupoids which fulfill condition \eqref{L}. 
We have a functorial diagram
\[
\xymatrix{
\LG(T) \ar[d]_{\iota} \ar[r]^{\Cstar} & \C^{\ast}(T) \\
\LCt(T) \ar[ur]^{\Ct^{\ast}} & 
}
\]

where $\iota$ denotes the inclusion functor and $\Ct^{\ast}$ is the functor constructed via representations given by \eqref{L}.  

% check: property (L) preserved by a T-arrow

\subsection*{Continuity for semi-groupoids}

% category of semi-Lie groupoids: functoriality needs (L) as an axiom
% introduce semi adiabatic deformation 
% continuity: ANS

In the remainder of this work we will be interested in a deformation semi-groupoid suitable for boundary value problems on Lie manifolds with boundary.
We will define the reduced $C^{\ast}$-algebra associated to this semi-groupoid and show that this yields a continuous field of $C^{\ast}$-algebras.
In order to give the reader a better appreciation of some of the difficulties involved in the study of the continuity of fields
of $C^{\ast}$-algebras associated to semi-groupoids we have discussed above the case of Lie groupoids. We are given the following data: A Lie manifold $(M, \A, \V)$ with boundary $(Y, \B, \W)$. Let $\G \rightrightarrows M$ denote an integrating $s$-connected Lie groupoid, i.e. $\A(\G) \cong \A$. Fix the generalized exponential map (see also \cite{LR}, \cite{BS}) $\Exp \colon \A(\G) \to \G$
as well as the exponential map $\exp \colon \A(\G) \to M$ induced by an invariant connection on $\A$. 
Consider the half space $\Abdy \subset \A$ which is defined as follows:
\[
\Abdy := \{v \in \A : \exp(-tv) \in M, \ t > 0 \ \text{small}\}. 
\]
This is the natural generalization of the half-space introduced in \cite{ANS} for the case of a compact manifold with boundary and trivial Lie structures.
We restrict the invariant connection $\nabla$ of $\A$ to $\Abdy$ and also obtain the restriction of the generalized exponential which we denote by the same symbol $\Exp \colon \Abdy \to \G$.
Then we define the deformation semi-groupoid $\Gbdy \rightrightarrows M \times I$ where $I = \Rr$ or $I = [0,1]$ and $I^{\ast} := I \setminus \{0\}$ as follows
\[
\Gbdy := \G \times (0,1] \cup \Abdy \times \{0\}.
\]
Note that a priori $\Gbdy$ has a natural semi-groupoid structure: the groupoid structure of $\G$, of $(0,1]$ viewed simply as a set
as well as the semi-groupoid structure of $\Abdy$, viewed as a bundle of half-spaces.
Note that $\Gbdy \subset \Gad$ where $\Gad$ is the adiabatic groupoid $\G \times (0,1] \cup \A \times \{0\}$.
By the local diffeomorphism property of $\Exp$ we can describe a smooth structure on $\Gad$. It is defined by glueing a neighborhood $\O$ of $\A \times \{0\}$ to $\G\times I^\ast$ via
\[
\O \ni (v, t) \mapsto \begin{cases} v, \ t = 0 \\
(\Exp(-tv), t), \ t > 0. \end{cases}
\]

Then the smooth structure of $\Gbdy \subset \Gad$ is the one induced by $\Gad$ with regard to the locally compact subspace topology, i.e.
$C_c^{\infty}(\Gbdy) := C_c^{\infty}(\Gad)_{|\Gbdy}$.
The next goal is to define a continuous field of $C^{\ast}$-algebras over the Lie semi-groupoid $\Gbdy$. 

\begin{Def}
The $C^{\ast}$-algebra associated to $\Gbdy$ is defined as the completion $C_r^{\ast}(\Gbdy) := \overline{C_c^{\infty}(\Gbdy)}^{\|\cdot\|}$.
We define the norm $\|\cdot\|$ as the reduced norm with regard to the representation $\tilde{\pi} := (\pi , \pi^{\partial})$ on the Hilbert space 
$\H := L^2(\G) \oplus L^2(\Abdy_{|Y})$. Here $\pi = (\pi_t)_{0 < t \leq 1}$ where for $0 < t \leq 1$ 
\[
\pi_t(f) \xi(\gamma) = \frac{1}{t^n} \int_{\G_{s(\gamma)}} f(\eta, t) \xi(\eta) \,d\mu_{s(\gamma)}(\eta).
\]

Define the representation $\pi_{0}^{\partial}$ on the Hilbert space $L^2(\Abdy_{|Y})$ by
\[
\pi_0^{\partial}(f) \xi(v) = \int_{\Abdy_{\tilde{\pi}(v)}} f(v - w) \xi(w)\,dw.
\]

\label{Def:CGbdy}
\end{Def}

We also introduce a $C^{\ast}$-algebra associated to the half-space $\Abdy$.
\begin{Def}
Define the reduced $C^{\ast}$-algebra of $\Abdy$ in terms of the completion $C_r^{\ast}(\Abdy) := \overline{C_c^{\infty}(\Abdy)}^{\|\cdot\|_{\tilde{\pi}_0}}$,
where $\tilde{\pi}_0 = (\pi_0, \pi_0^{\partial})$ is the representation of $C_c^{\infty}(\Abdy)$ on the Hilbert space $\H := L^2(\A) \oplus L^2(\Abdy_{|Y})$.
We define
\[
\pi_0(f) \xi(v) = \int_{\A_{\pi(v)}} f(v - w) \xi(w) \,dw
\]

and
\[
\pi_0^{\partial}(f) \xi(v) = \int_{\Abdy_{\tilde{\pi}(v)}} f(v - w) \xi(w) \,dw.
\]

\label{Def:CAbdy}
\end{Def}

Note that the above definition furnishes a field of $C^{\ast}$-algebras with $\varphi_t \colon C_r^{\infty}(\Gbdy) \to C_r^{\ast}(\Gbdy)(t)$ where
$C_r^{\ast}(\Gbdy_t) = C_r^{\ast}(\G), \ t \not= 0$ and $C_r^{\ast}(\Gbdy_0) = C_r^{\ast}(\Abdy)$. This leads us to the following result.
\begin{Thm}
Let $(M, \A, \V)$ be a decomposed Lie manifold with hypersurface $(Y, \B, \W)$ and integrating Lie groupoid $\G \rightrightarrows M$. 
The field $(C_r^{\ast}(\Gbdy), \{C_r^{\ast}(\Gbdy_{t}), \varphi_t\}_{t \in [0,1]})$ is a continuous field of $C^{\ast}$-algebras.
\label{Thm:CGbdy}
\end{Thm}

The continuity for the semi-groupoid $\Gbdy$ is a generalization of the result proven in \cite{ANS}, where the case of a compact manifold with boundary (and the trivial Lie structure of all vector fields) is studied. 
Since our case is vastly more general we give more details below. We will describe the strategy of the argument, highlighting the main differences to the argument in loc. cit. for our case of general decomposed Lie manifolds with boundary.

\begin{proof}
Note that the condition in the definition of the semi-algebroid $\Abdy$ only takes effect at the regular boundary stratum $Y$.
The difficulty in the proof of continuity of the field is the upper and lower semi-continuity of the field at $t = 0$. 
If we restrict the groupoid outside any tubular neighborhood of $Y$ in $M$ then $\Abdy$ is identical to $\A$.
For this restricted groupoid the argument for continuity of the field goes along the same lines as the proof given in \cite{ANS}.
Therefore we can without loss of generality focus on the case where $M$ is of cylinder type.
Since $M$ is assumed to be of cylinder type we identify $M \cong Y \times \overline{\Rr}_+$. Let $\H \rightrightarrows Y$ be a Lie groupoid such that $\A(\H) \cong \B$.
The groupoid $\G$ takes the form $\G = \H \times (\Rr_{+}^2 \cup \{0\} \times \Rr_{+}) \rightrightarrows Y \times \overline{\Rr}_+$
with smooth structure as defined in Section \ref{III}. We first define the auxiliary algebra $C_{tc}^{\infty}(\Abdy) := C_c^{\infty}(\A) \oplus C_c^{\infty}(\B \times \Rr_{+}^2)$, which is a 
dense $\ast$-subalgebra of $C_r^{\ast}(\Abdy)$. 
Choose a cutoff $\psi \in C_c^{\infty}(\G)$ such that $0 \leq \psi \leq 1$ and $\psi_{|U} \equiv 1$ for a neighborhood $M \subset U \subset \G$
for which there is a corresponding neighborhood of the zero section, $M \subset \mathcal{O} \subset \A(\G)$ for which $\Exp \colon \mathcal{O} \iso \supp \psi \subset \G$
is a diffeomorphism. Throughout we will make use of the lifting of an element $f \in C_c^{\infty}(\A)$, or $f \in C_c^{\infty}(\Abdy)$ to an element of $C_c^{\infty}(\Gbdy)$, defined by
\begin{align}
\tilde{f}(\gamma, t) &:= \psi(\gamma) f\left(-\frac{\Exp^{-1}(\gamma)}{t}\right). \tag{$l$} \label{l}
\end{align}

We first show the lower semi-continuity of the field, i.e.
\begin{align}
& \liminf_{t \to 0} \|\varphi_t(a)\| \geq \max\{\|\pi_0(a)\|, \|\pi_0^{\partial}(a)\|\}. \tag{$lsc$} \label{lower}
\end{align}

The proof of lower semi-continuity is fascilitated by a reduction of the proof of lower semi-continuity of representations of $C_c^{\infty}(\Gbdy)$ and $C_c^{\infty}(\Had \times \overline{\Rr}_+)$.
Making use of the generalized exponential $\Exp_{\partial} \colon \B \to \H$, we can introduce a lifting \eqref{l} also for elements $f \in C_c^{\infty}(\B \times \Rr_{+})$ to $C_c^{\infty}(\Had \times \Rr_{+})$ and, abusing notation, we also
denote by $\tilde{f}$. Fixing a Haar system $(\mu_{x, t})_{(x, t) \in M \times I}$, introduce the norms $\|\cdot\|_{\infty, t}$ on $\Gbdy$ by
\[
\|g\|_{\infty, t}^2 := \sup_{x \in \Gop} \|g\|_{L^2(\Gbdy_{x,t}, \mu_{x,t})}.
\]

Abusing notation again, we use the same symbols $\|\cdot\|_{\infty, t}$ for the corresponding norms on $C_c^{\infty}(\Had \times \Rr_{+})$, where
we fix a Haar system $\mu_{x,t}^{\partial}$ on $\Had$ and the standard Lebesgue measure on $\Rr_{+}$. 
Define the representation $\rho_t \colon C_c^{\infty}(\Abdy) \to \L(L^2(\G))$ by 
\[
(\rho_t f) \xi(\gamma) = \int_{\G_{s(\gamma)}} \tilde{f}(\gamma \eta^{-1}, t) \xi(\eta)\,d\mu_{s(\gamma)}(\eta).
\]

A straightforward generalization of \cite{ANS}, Proposition 2.22 yields the density of $C_c^{\infty}(\Abdy)$ in $C_c^{\infty}(\Gbdy)$, which implies that it is sufficient to show the estimates
\begin{align}
&\liminf_{t \to 0} \|\rho_t(\tilde{f} + \tilde{K}\| \geq \|\pi_0(f)\|, \ f \oplus K \in C_{tc}^{\infty}(\Abdy), \label{aux1} \\
&\liminf_{t \to 0} \|\rho_t(\tilde{f} + \tilde{K}\| \geq \|\pi_0^{\partial}(f \oplus K)\|, f \in C_c^{\infty}(\A), \ K \in C_c^{\infty}(\B \times \Rr_{+}^2). \label{aux2}
\end{align}

In order to show \eqref{aux1} we check that % proof of aux1
\begin{align*}
& \|\pi_t(\tilde{f} + \tilde{K})\| = \sup\left\{\left\|\frac{1}{t^n} \int_{\G_{\bullet}} (\tilde{f}(\bullet \cdot \eta^{-1}, t) + \tilde{K}(\bullet \cdot \eta^{-1}, t) g(\bullet \cdot \eta^{-1}, t) \,d\mu_{\bullet}(\eta)\right\|_{\infty, t} : \|g\|_{\infty} \leq 1 \right\}, \\
& \|\pi_0(f)\| = \sup\left\{\left\| \int f(v, 0) g(\bullet - v, 0)\,dv \right\| : \|g\|_{\infty} \leq 1\right\}.
\end{align*}

At this point we recall the structure of the groupoid $\Gbdy$ as $\H \times (\Rr_{+}^2 \cup \{0\} \times \Rr_{+})$ and note that the set of $g$
for which $g(\gamma, t) = 0$ with $s(\gamma) = (x', 0) \in Y \times \overline{\Rr}_{+} \cong M$ is dense in $\{g \in C_c^{\infty}(\Gbdy) : \|g\|_{\infty} \leq 1\}$.
The weak convergence of $\tilde{K}$ to zero yields for $g \in C_c^{\infty}(\Gbdy)$
\[
\lim_{t \to 0} \left\|\frac{1}{t^n} \int (\tilde{f}(\bullet \cdot \eta^{-1}, t) + \tilde{K}(\bullet \cdot \eta^{-1}, t)) g(\bullet \cdot \eta^{-1}, t)\,d\mu_{\bullet}(\eta)\right\|_{\infty, t} = \left\|\int f(v, 0) g(\bullet - v)\,dv\right\|_{\infty, 0}.
\]

The equality \eqref{aux1} follows. We prove \eqref{aux2} by fixing $(a_t)_{t \in I}$ such that $a_t \to 0$ for $t \to 0$ and $\frac{a_t}{t} \to \infty$ for $t \to 0$.
Define representations of $f \in C_c^{\infty}(\Gbdy), \ K \in C_c^{\infty}(\B \times \Rr_{+}^2)$ on $C_c^{\infty}(\Had \times \Rr_{+}^2)$ via
\[
\eta_t(f) g(\gamma, t, b) = t^{n+1} \int_{[0, \frac{a_t}{t}]} \int_{\G_{s(\gamma)}} f(\gamma \eta^{-1}, tb, ta, t) g(\eta, t, a) \,d\mu_{s(\gamma)}(\eta) \,da
\]

for $t \not= 0$ and $b \in \left[0, \frac{a_t}{t}\right]$. As well as
\begin{align*}
\eta_0(f) g(v, 0, b) &= \int_{\B_x \times \Rr_{+}} f(v - w, b- a) g(w, 0, a) \,dw \,da, \\
\eta_t(K) g(\gamma, t, b) &= t^{-n + 1} \int_{[0, \frac{a_t}{t}]} \int_{\G_{s(\gamma)}} K(\Exp_{\partial}^{-1}(\gamma \eta^{-1} t^{-1}, b, a) g(\gamma \eta^{-1}, t, a)\,d\mu_{s(\gamma)}(\eta) \,da. \\
\eta_0(K) g(v, 0, b) &= \int_{\B_x \times \Rr_{+}} K(v - w, b, a) g(w, 0, a)\,dw \,da. 
\end{align*}

Denote by $P_t$ the operator given by multiplication with the characteristic function $\H \times [0, a_t]^2$, where $[0, a_t]^2$ is the pair groupoid.
Also denote by $D_t$ the dilation by $t$ operator. 
Then we have $\|P_t \pi_t(f) P_t\| = \|D_t P_t \pi_t(f) P_t D_{t^{-1}}\| = \sup\{\|\eta_t(f) g\|_{\infty, t} : \|g\|_{\infty} \leq 1\}$.
Hence $\|\pi_0^{\partial}(f \oplus K)\| = \sup\{\|\eta_0(f \oplus K) g\|_{\infty, 0} : \|g\|_{\infty} \leq 1\}$. 
Let $g \in C_c^{\infty}(\Gbdy)$ and note that for $t$ small we can without loss of generality assume that $g$ takes the form
$g_0 \left(-\frac{\Exp_{\partial}(\gamma \eta^{-1})}{t}, t, a\right)$ with $g_0 \in C_c^{\infty}(\B \times [0,1] \times \overline{\Rr}_{+})$. 
Thence
\[
\lim_{t \to 0} \|(\eta_t(\tilde{f} + \tilde{K})g\|_{\infty, t} = \|(\eta_0(f \oplus K) g\|_{\infty, 0} 
\]

which implies \eqref{aux2}. From \eqref{aux1} an \eqref{aux2} we obtain the lower semi-continuity \eqref{l}. 
The upper semi-continuity is the inequality
\begin{align}
& \limsup_{t \to 0} \|\varphi_t(a)\| \leq \max\{\|\pi_0(a)\|, \|\pi_0^{\partial}(a)\|\}. \tag{$usc$} \label{upper}
\end{align}

This follows by the density result of \cite{ANS}, Proposition 2.22. The remainder of the argument is analogous to loc. cit. and 
we omit the details. The estimates \eqref{lower} and \eqref{upper} together imply the continuity of the field of $C^{\ast}$-algebras.
\end{proof}

% Main Theorem: The field of $C^{\ast}$-algebras $(C_r^{\ast}(\Gbdy), \varphi)$ is continuous.
% beginning of proof:
% condition in the definition of \Abdy only takes effect at the "boundary" Y
% hence outside of a tubular neighborhood \Abdy looks just like \A
% therefore the continuity is only an issue for the cylinder type Lie groupoid \G
% hence wlog the proof suffices for the case of cylinder type Lie groupoids
% follow ANS with some difficulties...
% Setup: introduce lifting trick
% Assume \G of cylinder type
% upper & lower semi-continuity

\subsection*{Functional calculus}

On a given decomposed Lie manifold we define a functional calculus taking values in the reduced $C^{\ast}$-algebra of the 
deformation semi-groupoid considered in the previous section.
Denote by $\P$ the set of functions in the Schwartz class $\S(\Rr)$ which have compactly supported Fourier transform.

\begin{Thm}
Let $(M, \A, \V)$ be a decomposed Lie manifold with hypersurface $(Y, \B, \W)$ with corresponding integrating groupoid $\G \rightrightarrows M$. 
Denote by $\Dd := (\Dirac_{x, t})_{(x, t) \in M \times I}$ an equivariant family of geometric Dirac operators associated to $\nablaslash^{W}$ on $\Gbdy$.
Then there exists a ring homomorphism
\[
\Psi_{\Dd} \colon C_0(\Rr) \to C_r^{\ast}(\Gbdy)
\]

which is compatible with the representations $(\pi, \pi^{\partial})$, i.e. given $\pi = (\pi_t)_{t \in I}$  and $\pi^{\partial}$ we have
\begin{align}
\pi_{x,t}(\Psi_{\Dd}(f)) &= f(\Dirac_{x,t}), \ f \in \P, \label{R1}
\end{align}

as well as for $\pi_0^{\partial} \colon C_c^{\infty}(\Abdy) \to \L(L^2(\G_{|Y}))$ 
\begin{align}
\pi_0^{\partial}(f) \xi(\gamma) &= \int_{\Abdy_{r(\gamma)}} f(w) \psi(w) \xi(\Exp^{-1}(\gamma) - w)\,dw \label{R2}
\end{align}

with 
\[
\pi_0^{\partial}(\Psi_{\Dd}(f)) = f(\Dirac_0), \ f \in \P. 
\]
\label{Thm:fctcalc}
\end{Thm}

\begin{proof}
Denote by $e^{i \tau \Dirac_{x,t}}$ the solution operator to the wave equation for $\Dirac_{x,t}$. For $f \in \P$ define via the functional calculus (cf. \cite[Section 3.C]{Ch})
\[
f(\Dirac_{x,t}) = \frac{1}{2\pi} \int \hat{f}(\tau) e^{i \tau \Dirac_{x,t}} \,d\tau. 
\]

By the estimates in the proof of Proposition 7.20 in \cite{R}, we obtain that $f(\Dirac_{x,t})$ is a smoothing operator of finite propagation speed. 
The family $f(\Dirac_{x,t})$ is $\Gbdy$-invariant (i.e. with regard to the natural action of $\Gad \supset \Gbdy$). 
Take the reduced kernel $k^f$ over $\Gad$ and denote its restriction to $\Gbdy$ by the same symbol. By finite propagation speed and the equivariance 
it follows that $k^f$ is compactly supported, see also \cite{R0}. Define $\Psi_{\Dd}(f)$ via the assignment $\gamma \mapsto k_{s(\gamma)}^f$.
The latter assignment furnishes a ring homomorphism, cf. the proof of Proposition 9.20 in \cite{R}. The compatibility \eqref{R1} follows
since $k^f$ is a reduced convolution kernel and \eqref{R2} follows by the definition of $\Dirac_0$ on $\Abdy$.
We obtain the $L^2$-action of $\Psi_{\Dd}$:
\begin{align*}
f(\Dirac_{x,t}) g(\gamma) &= \pi_{x,t}(\Psi_{\Dd}(f)) g(\gamma) \\
&= (\Psi_{\Dd}(f) \ast g)(\gamma), \ t > 0
\end{align*}

as well as
\begin{align*}
f(\Dirac_0) g(v) &= \pi_0^{\partial}(\Psi_{\Dd}(f)) g(v) \\
&= (\Psi_{\Dd}(f) \ast g)(v), \ t = 0. 
\end{align*}

Altogether we have shown that $\Psi_{\Dd} \colon \P \to C_c^{\infty}(\Gbdy)$ is a ring homomorphism that is compatible with $(\pi, \pi_0^{\partial})$. 
Since $\P$ is dense in $C_0(\Rr)$ and $C_r^{\ast}(\Gbdy)$ is defined as the completion with regard to the representations $(\pi, \pi_0^{\partial})$,
we obtain by the $L^2$-spetral theorem that the map $\P \to C_c^{\infty}(\Gbdy)$ is continuous with regard to the $C_0(\Rr)$-norm.
Hence $\Psi_{\Dd}$ extends continuously to a ring homomorphism $C_0(\Rr) \to C_r^{\ast}(\Gbdy)$. 
\end{proof}

\section{Index formula}

We give the proof of a generalized APS-type index formula on a decomposed spin Lie manifold for a geometric Dirac operator subject to local and non-local boundary conditions.
The local boundary conditions are posed on a stratum of the cylinder part of the decomposed Lie manifold. 
On the other hand we have non-local APS boundary conditions which are implicit on the complementary part of the manifold, where
the boundary is pushed to infinity. The technique we employ to prove the index formula makes use of the theory outlined in the previous sections:
The functional calculus from section \ref{IV} is used, combined with the rescaling bundle technique from \cite{BS}, in order to derive an index formula
on standard spin Lie manifolds.

Consider a decomposed Lie manifold $(M, \A, \V)$ with hypersurface $(Y, \B, \W)$, i.e. $M = M_1 \cup M_2$ where $M_2$ is of cylinder type.
We study a geometric admissible Dirac operator defined on $M$. The operator $D$ is assumed to be decomposed into two geometric admissible Dirac operators, i.e. $D_{|M_1} = D_1, \ D_{|M_2} = D_2$. 
The operator $D_2$ is of model type and admits local boundary conditions, while $D_1$ is a graded geometric admissible Dirac operator on the complement $M_1$.
Fix the notation $\Abdy \to M$ for the semi-Lie algebroid as introduced in the previous section. 
With the help of the functional calculus introduced in the previous section and the rescaling argument from \cite{BS}, 
we can separate the index calculation for the two cases of $D_1$ and $D_2$.

\subsection*{Renormalizable Lie manifolds}

Similarly as in \cite{BS} we define a renormalized trace for renormalizable Lie manifolds.
Denote by $\Omega^1(\A)$ the $1$-forms on $\A$ and by $\dot{C}^{\infty}(M, \Omega^1(\A))$ the smooth sections vanishing to all orders at the boundary strata of $M$.
% renormalizable Lie manifolds: general theory goes here...
\begin{Def}
A Lie manifold $(M, \V, \A)$ is called \emph{renormalizable} if there is a  functional $\TrV \colon C^{\infty}(M, \OmegaV^1) \to \Cc, \ f \mapsto \avintV f$ with the following properties:

\emph{1)} The integral $ \int_{M} f$ exists for $f\in  \dot{C}^\infty(M, \OmegaV^1) $, and  the functional $\TrV$ is a linear extension. 

\emph{2)} There is a minimal $k \in \Rr$ such that $G(f)(z)=\int_M \rho^z f $ defines a function  $G(f)$ holomorphic on $\Re (z) > k-1$, which extends meromorphically to $\Cc$.

Then we can define 
\[
\davintVM f : \ \text{regularized value (zero order Taylor coefficient) at} \ z = 0 \ \text{of} \ G(f). 
\]
\label{Def:renorm}
\end{Def}

\begin{Lem}
Let $(M, \A, \V)$ be a renormalizable Lie manifold with boundary $(Y, \B, \W)$. Then $(Y, \B, \W)$ is renormalizable as well. 
\label{Lem:renorm}
\end{Lem}

\begin{proof}
We construct the functional $H \colon C^{\infty}(Y, \Omega^1(\B)) \to \Cc, \ f \mapsto \avintWY f$ as a linear extension.
By Theorem \ref{Thm:tubnbhd} there is a global tubular neighborhood $Y \hookrightarrow \U \hookrightarrow M$ such that the boundary defining function 
$\rho_Y$ of $Y$ is expressed in the coordinates of $\U$, i.e. 
\[
(\rho_Y \circ \nu)(x_1, x') = x_1, \ (x_1, x') \in Y_{(\epsilon)}
\]

where $\nu \colon Y_{(\epsilon)} = Y \times (-\epsilon, \epsilon) \iso \U$. 
Let $\rho$ denote the boundary defining function of $Y$, i.e.
\[
\rho := \prod_{F \in \F_1(Y)} \rho_F. 
\]

Fix the local coordinates over $\U$ by $x' = (x_2, \cdots, x_{n-1})$, then $\rho = x_2 x_3 \cdots x_{n-1}$. 
Since the degneracy index $k$ of $M$ is finite, there is $l \geq k$ such that 
\[
H(f)(z) = \int_Y \rho^z f
\]

is holomorphic in $\{\Re(z) \geq l-1\}$ and extends meromorphically to $\Cc$. Then $\avintWY f = \mathrm{Reg}_{z=0} H(f)(z)$ is defined and finite.
Hence $(Y, \B, \W)$ is renormalizable. 
\end{proof}

\begin{Exa}
We refer to \cite{BS}, Section 5 for a large class of examples of renormalizable Lie manifolds where it is shown that so-called \emph{exact} Lie manifolds are renormalizable. 
Here a Lie structure $\V$ of a Lie manifold $(M, \V, \A)$ is called \emph{exact} if, near each face with boundary defining function $x_1$,  the Lie structure is generated by vector fields of the form $\left\{x_1^{k_1}\partial_{x_1},  \ldots,  x_1^{k_n}\partial_{x_n}\right\}$ for arbitrary $\{k_l : \ 1 \leq l \leq n\}$. 
\label{Exa:exact}
\end{Exa}

\subsection*{Cylinder type index formula}

Fix the geometric admissible Dirac operator $D$ on a Lie manifold $(M, \A, \V)$ of cylinder type with boundary $(Y, \B, \W)$. The Dirac operator $D = D^W$ is defined via a Clifford module $W \in \Cl(\Abdy)-\mathrm{mod}$. After an application of Lemma \ref{Lem:bdysym} without loss of generality $D$ is assumed to be of model type, i.e.
\begin{align}
D &= \begin{pmatrix} i \partial_u & i D_{\partial}^{-} \\
-i D_{\partial}^{+} & -i \partial_u \end{pmatrix} \label{normal}
\end{align}

where $D_{\partial}$ is an admissible geometric Dirac operator on $(Y, \B, \W)$, i.e. 
\[
D_{\partial} = \begin{pmatrix} 0 & i D_{\partial}^{-} \\
-i D_{\partial}^{+} & 0 \end{pmatrix}
\]

for a fixed admissible connection $\nabla^{W_{\partial}}$ and $W_{\partial} \in \Cl(\B)-\mathrm{mod}$.
Note that $D^{\ast} D = -\partial_u^2 + D_{\partial}^2$. 
Let $\varphi \in \Gamma(W)$ be a smooth section such that $\varphi = \varphi^{+} \oplus \varphi^{-}$ corresponding to the grading
$W = W^{+} \oplus W^{-}$. Denote by $B^{\pm}$ the local boundary condition $(\varphi_{|Y})^{\pm} = 0$. 
Note that $(D, B^{+})^{\ast} = (D, B^{-})$. We study the \emph{induced boundary conditions} for \emph{(a)} the case $D^{\ast} D$, where
$\varphi^{+}(0, y) = 0$ and $(\partial_u \varphi^{-} + D_{\partial}^{+} \varphi^{+})_{u=0} = 0$. 
Since $D_{\partial}^{+}$ is a tangential operator it follows that $D_{\partial}^{+} \varphi^{+} = 0$. Therefore in case \emph{(a)} the local boundary
conditions are subdivided into Dirichlet and Neumann condition
\begin{align}
& \varphi^{+}(0, y) = 0, \ \tag{$Da$} \label{Da} \\
& \partial_u \varphi^{-}(0, y) = 0. \ \tag{$Na$} \label{Na}
\end{align}

The induced boundary conditions in the case \emph{(b)} of $D D^{\ast}$ are for $\psi \in \Gamma^{\infty}(W)$ given by 
\begin{align}
& \partial_u \psi^{+}(0,y) = 0, \ \tag{$Nb$} \label{Nb} \\
& \psi^{-}(0, y) = 0. \tag{$Db$} \label{Db}
\end{align}

Define the \emph{renormalized index} as defined via the renormalized super trace 
\[
\indV(D) = \lim_{t \to \infty} \TrV_s(e^{-t D^{\ast} D}) - \TrV_s(e^{-t D D^{\ast}}).
\]

We have the following result concerning the renormalized index of a model form Dirac operator.
\begin{Thm}
Let $(M, \A, \V)$ be a renormalizable, spin Lie manifold of cylinder type with boundary $(Y, \B, \W)$ and let $D$ be an admissible graded Dirac operator of model type \eqref{normal}. 
Then the renormalized index of $D$ is
\[
\indV(D) = \lim_{t \to \infty} \TrV_s(e^{-t D^2}) = \mp \frac{1}{2} \indW(D_{\partial}) 
\]

for $D$ subject to the boundary condition $B^{\pm}$. 
\label{Thm:cyl}
\end{Thm}

\begin{proof}
Set $K_{\pm}(t, u) := \TrV_s(e^{-t D^{\ast} D} - e^{-t D D^{\ast}})$ for the density with regard to the boundary condition $B^{\pm}$. 
Fix the volume forms $\mu$ on $M$ and the induced volume form $\mu_{\partial}$ on $Y$, by Proposition \ref{Prop:measured}. 
Write $K_{+}(t, u) = K_1(t, u) - K_2(t, u)$, then we have
\begin{align*}
K_{\pm}(t, u) &:= \TrV(e^{-t D^{\ast} D} - e^{-t D D^{\ast}}) = K_1^{\pm}(t, u) - K_2^{\pm}(t, u) \\
&= \davintWY \int_{\overline{\mathbb{R}}_{+}} \tr \ \tilde{K}_1(t, v, v, y, y) \, dv \,d\mu_{\partial}(y) - \davintWY \int_{\overline{\mathbb{R}}_{+}} \tr \ \tilde{K}_2(t, v, v, y, y) \,dv \, d\mu_{\partial}(y). 
\end{align*}

We next calculate the kernels with regard to \eqref{Da}, \eqref{Na} separately. We make use of the calculations in \cite{MS}, see also \cite{Z} to obtain the following formulae.
Case \eqref{Da} and \eqref{Db}:
\begin{align*}
\tilde{K}_1(t, v, w, y, z) &= \frac{1}{\sqrt{4\pi t}} \left\{\exp\left(\frac{(v - w)^2}{4t} - \exp\left(- \frac{(v + w)^2}{4t} \right) \right) \right\} e^{-t D_{\partial}^{+} D_{\partial}^{-}}(t, y, z). \\ 
\tilde{K}_2(t, v, w, y, z) &= \frac{1}{\sqrt{4 \pi t}} \left\{\exp\left(-\frac{(v-w)^2}{4t}\right) - \exp\left(-\frac{(v + w)^2}{4t} \right) \right\} e^{-t D_{\partial}^{-} D_{\partial}^{+}}(t, y, z). 
\end{align*}

Case \eqref{Na} and \eqref{Nb}: 
\begin{align*}
\tilde{K}_1(t, v, w, y, z) &= \frac{1}{\sqrt{4\pi t}} \left\{\exp\left(-\frac{(v-w)^2}{4t} \right) + \exp\left(-\frac{(v +w)^2}{4t} \right)\right\} e^{-t D_{\partial}^{+} D_{\partial}^{-}}(t,y,z), \\
\tilde{K}_2(t, v, w, y, z) &= \frac{1}{\sqrt{4 \pi t}} \left\{\exp\left(-\frac{(v-w)^2}{4t} \right)  + \exp\left(-\frac{(v+w)^2}{4t}\right)\right\} e^{-t D_{\partial}^{-} D_{\partial}^{+}}(t, y, z). 
\end{align*}

With the help of these explicit formulae we obtain the trace density of $e^{-t D^{\ast} D}$ subject to $B^{+}$ 
\begin{align*}
K_1(t, u) &= \davintWY \int_{\overline{\mathbb{R}}_{+}} \tr \ K_1(t, v, w, y, y) \,dv \,d\mu_{\partial}(y) \\
&= \frac{1}{\sqrt{4\pi t}} \TrW(e^{-t D_{\partial}^{-} D_{\partial}^{+}}) \left\{1 - \exp\left(-\frac{u^2}{t} \right)\right\} + \frac{1}{\sqrt{4\pi t}} \TrW(e^{-t D_{\partial}^{+} D_{\partial}^{-}}) \left\{1 + \exp \left(-\frac{u^2}{t} \right) \right\}.
\end{align*}

Similarly for $K_2(t, u)$, the density of $e^{-t D D^{\ast}}$. 
Altogether we obtain
\begin{align*}
K_{+}(t, u) &= \frac{e^{-\frac{u^2}{t}}}{\sqrt{\pi t}} (\TrW(e^{-t D_{\partial}^{+} D_{\partial}^{-}}) - \TrW(e^{-t D_{\partial}^{-} D_{\partial}^{+}})).
\end{align*}

By an application of the McKean-Singer formula (see also \cite{MS}) we have
\[
\int_{0}^{\infty} K_{+}(t, u) \,du = -\frac{1}{2} \indW(D_{\partial}).
\]

For the case $B^{-}$ we obtained
\[
 \int_{0}^{\infty} K_{-}(t, u) \,du = \frac{1}{2} \indW(D_{\partial})
\]

by an analogous calculation.
\end{proof}

\subsection*{Rescaling for semi-groupoids}

% rescaling proof for adiabatic semi-groupoid

For any decomposed Lie manifold our aim is to find a topological interpretation of the renormalized index. We prove a generalized APS-type index formula. The formula for the class of non-compact manifolds involves a local contribution, depending on the metric and
a non-local contribution which also depends on restrictions of the Dirac operator to the boundary (the so-called indicial symbol). 
The formula holds a priori for any Dirac operator which is not required to be Fredholm and the renormalized index may take non-integer values.
Then we prove another result which gives criteria for the equality of the Fredholm index with the renormalized index as well as the conditions for the 
Dirac operator to be Fredholm. The main new feature of our formula over the special case investigated already in \cite{BS} is that it extends to the case of local boundary conditions
in a singular setting. For a given renormalizable Lie manifold $(M, \A, \V)$ and graded Dirac operator $D$ we fix the definition of the renormalized $\eta$-invariant $\etaV(D) := \frac{1}{2} \int_{0}^{\infty} \TrV_s([D, D e^{-t D^2}]) \,dt$.

\begin{Thm}
Let $(M, \A, \V)$ be a decomposed renormalizable spin Lie manifold with hypersurface $(Y, \B, \W)$ such that $M = M_1 \cup_Y M_2$.
Denote by $D = D^W$ a geometric Dirac operator for an admissible $\Abdy$-connection $\tilde{\nabla}^W$ with
$D_{|M_1} = D_1$, $D_{|M_2} = D_2$, where $D_2$ is of model type on the cylinder $M_2$ and $D_1$ is a geometric Dirac operator over $M_1$.
Then subject to the boundary condition $B^{\pm}$ we have
\begin{align}
& \indV(D) = \davintVMc \hat{A}(\nabla_1) \wedge \exp F^{W_2 / S} \,d\mu + \etaVc(D_1) \notag \\
& \mp \frac{1}{2} \davintWY \hat{A}(\nabla_{\partial}) \wedge \exp F^{W_{\partial} / S} \,d\nu_{\partial} + \etaW(D_{\partial}). 
\end{align}
\label{Thm:index}
\end{Thm}

\begin{proof}
Let $W \to M$ be the $\Cl(\Abdy)$ module compatible with the Clifford action. 
Denote by $\hom(W) \to M$ the bundle with fibers $\hom(W)_x \cong \hom(W_x, W_x) \cong \Cl(\Abdy_x \otimes \Cc) \otimes \End_{\Cl}(W_x)$ and
by $\Hom(W) \to M$ the homomorphism bundle. We have $\Hom(W)_{|\Abdy} \cong \Cl(\Abdy \otimes \Cc) \otimes \End_{Cl}(W)$.
Fix the Lie groupoid $\G \rightrightarrows M$ with $\A(\G) \cong \A$ as given in Theorem \ref{Thm:decomposed}.
We can lift $\Hom(W)$ to a bundle over $\Gbdy$ using the range and source map, i.e. by forming the pullback bundle
$r^{\ast}(\Hom(W)) \otimes s^{\ast}(\Hom(W)^{\ast})$. Since no confusion will arise we denote this lifted bundle by the same symbol $\Hom(W) \to \Gbdy$. 
Consider the geometric Dirac operator $\Dirac$ on $\G$ such that $\varrho(\Dirac) = D$ (cf. \cite{LN}). 
Then the reduced heat kernel $k_t$ on $\Gbdy$ is well-defined and we have the formal heat kernel expansion (cf. \cite{So})
\[
e^{-t \Dirac^2} \sim (4\pi)^{-\frac{n}{2}} \sum_{i=0}^{\infty} a_i(x) t^{\frac{n}{2} - i}, \ t \to 0^+, \ x \in M \subset \G
\]

for $a_i \in \Gamma^{\infty}(\Cl(\Abdy^{\ast}) \otimes \End_{\Cl(\Abdy^{\ast})}(W))$. 
We have the filtration by Clifford degree $\Cl_0 \subseteq \Cl_1 \subseteq \cdots \subseteq \Cl_n(\Abdy \otimes \Cc)$. 
Topologically $\Gbdy$ is a manifold with corners in its own right and we can view $\Abdy$ as a particular boundary stratum.
By restricting the generalized exponential mapping $\Exp \colon\A \to \G$ to the half space $\Abdy$ we obtain a tubular neighborhood and a 
normal direction (cf. \cite{LR}). Denote by $N = \partial_t$ the corresponding normal vector field in $\Gbdy$. Define 
\[
\DV := \{u \in C_c^{\infty}(\Gbdy, \Hom(W)) : \nabla_N^p u_{|\Abdy} \in C^{\infty}(\Abdy, \Cl_{n-p} \otimes \End_{\Cl}(W)), 0 \leq p \leq n\}.
\]

By a proof analogous to \cite{BS}, Proposition 6.4. we obtain that the filtration $\{Cl_j\}$ can be extended by parallel transport
along $\nabla_N$ to a neighborhood of $\Abdy$ inside $\Gbdy$. Denote by $\{\tilde{\Cl}_j\}$ the extended filtration. Then we have the alternative description 
\[
\DV = \{u \in C_c^{\infty}(\Gbdy, \Hom(W)) : u = \sum_{j=0}^n t^{n-j} u_j + t^{n+1} u', \ \text{near} \ \Abdy\}
\]

where $u_j \in C_c^{\infty}(\Gbdy, \tilde{\Cl}_{n-j} \otimes \End_{\Cl}(W))$ and $u' \in C_c^{\infty}(\Gbdy, \Hom(W))$.
By the Serre-Swan theorem there is a \emph{rescaling bundle} $\Ee \to \Gbdy$ such that $C_c^{\infty}(\Gbdy, \Ee) = i_{\Cl}^{\ast} \D$.
Here $i_{\Cl} \colon \Ee \to \Hom(W)$ is a bundle map which is an isomorphism over the interior $\Gbdy_{(0,1]}$. 
It is not hard to check that we obtain a canonical isomorphism of Clifford algebras $\Ee_{\Abdy} \cong \Lambda \Abdy^{\ast} \otimes \End_{\Cl}(W)$.
The structure of $\Ee$ will make sure that we extract the correct coefficient in the formal heat kernel expansion.
Set $\Dd := (t \Dirac_x)_{(x,t) \in M \times I}$ for the family of Dirac operators over $\Gbdy$.
Consider $f(x) = e^{-x^2}$ and assume for technical reasons that $f$ is convolved with a function $\lambda$ which has Fourier transform with large compact support.
Denote this convolution by $\tilde{f} := f \ast \lambda$. By the definition of the functional calculus in Theorem \ref{Thm:fctcalc} we obtain that
$\Psi_{\Dd}(\tilde{f}) = t^n k_{t^2}$ as an element of $C_r^{\ast}(\Gbdy)$. 
This follows from the action of the functional calculus $f(\Dirac_{x,t}) g(\gamma) = \pi_{x,t}(\Psi_{\Dd}(f) \ast g)(\gamma)$ for $t > 0$ which yields
\[
f(t \Dirac) g(\gamma) = \int_{\G_{s(\gamma)}} \Psi_{\Dd}(f)(\gamma \eta^{-1}) g(\eta) t^{-n} \,d\mu_{s(\gamma)}(\eta).
\]

The scaling factor $t^{-n}$ enters by a choice of Haar system as in \cite[(6.8)]{LR}.
Set $l_t := \Psi_{\Dd}(f)_{|\G_{\Delta}}$, then $l_t(\gamma) = t^n k_{t^2}(\gamma)$ for $t \not= 0$.
Define the diagonal $\G_{\Delta} := \{\gamma \in \Gbdy : s(\gamma) = r(\gamma) \} \subset \Gbdy$. 
For $t \not= 0$ the supertrace functional maps $\tr_s \colon C_c^{\infty}(\G_{\Delta}, \Ee_{\Delta}) \to t^n C_c^{\infty}(\G)$, cf. \cite[Proposition 11.4]{R}.
Therefore $t^{-n} \tr_s(l_t)$ extends smoothly to $t = 0$ and hence $t^{-n} \tr_s(l_t) = \tr_s(l_0) + o(t)$. 
We are therefore reduced to calculate $\tr_s(l_0)$. Note that since $l_0$ lives on the semi-algebroid $\Abdy$ we have to be careful
to incorporate the boundary conditions posed on $Y$.
First note that $\Abdy(\G)_{|M_2} \cong \Abdy(\G_2)$ is the semi-algebroid corresponding to the cylinder type manifold $M_2$.
On the other hand $\Abdy(\G)_{|M_1} \cong \A(\G_1)$ is the Lie algebroid corresponding to the Lie manifold (without boundary) $M_1$.
In the second case the calculation goes analogous to the proof given in \cite{BS}. We repeat the key steps of the argument for completeness before dealing with the first case.
Denote by $\Phi \colon U^{ad} \times V \iso \Rr^n \times \Rr^m \times \Rr$ a diffeomorphism where we write $U^{ad} = U \times \Rr$.
The coordinates induced by $\Phi$ should form a \emph{parametrization} of the adiabatic groupoid of $\G_1$, cf. \cite{LR} and \cite[p.145]{NWX}.
Let $\alpha_x = \alpha_{\A_x(\G_2) \cap V}$ denote the restriction of the generalized exponential map $\Exp_x$ on the fiber $\G_x$.
If $V$ is chosen sufficiently small we can fix a local geodesic coordinate system $\alpha_x(\gamma) = (a_1, \cdots, a_m) =: a$. 
Let $\Phi_{x,t}$ be the restriction of $\Phi$ to $V \times \{x\} \times \{t\}$. An elementary calculation yields $\Phi_{x,t}(\eta) = \frac{1}{t} (\alpha_x(\eta) - a)$.
Applying the Lichnerowicz formula on the complete Riemannian manifold $(\G_x, g_x)$ and taking the limit as $t \to 0$ we obtain (cf. \cite{BS})
\[
\Dirac_{x,0}^2 = -\sum_{i} \left(\partial_i^x + \frac{1}{4} \sum_{j} R_{ij}^x a_j\right)^2 + \sum_{i < j} F^{W_x / S}(e_i, e_j)(a_j)(a_j). 
\]

The differential equation of the heat kernel of $\Dirac_{x,0}^2$ is a harmonic oscillator. Applying \cite{BGV} we have the solution
in terms of a formal power series in the scalar curvature $R_{ij}^x$ and the exponential of the twisting bundle $\exp F^{W_x / S}$.
By the $\G$-invariance of the curvature tensor as well as the twisting curvature and the Lichnerowicz theorem
for Lie manifolds given in \cite[Theorem 2.4]{BS}, it follows from \cite{BGV}, p. 164 and \cite{R}, Proposition 12.25, 12.26 the integrand $\Aroof \wedge \exp F^{S/W}$ 
in the trace formula. Thus we have shown that
\[
\lim_{t \to 0^+} \TrsV(e^{-tD_1^2}) = \davintVMc \hat{A}(\nabla_1) \wedge \exp F^{W / S} \,d\mu.
\]
To obtain the limit $t \to \infty$ we notice that
\[
\lim_{t \to \infty} \TrsV(e^{-t D_1^2}) - \lim_{t \to 0^{+}} \TrsV(e^{-t D_1^2}) = \int_{0}^{\infty} \partial_t \TrsV(e^{-t D_1^2})\,dt.
\]

Observe that $\partial_t \TrsV(e^{-tD_1^2}) = \TrsV(\partial_t e^{-t D_1^2})$. Setting $\etaVc(D_1) := \frac{1}{2} \int_{0}^{\infty} \TrsV(D_1^2 e^{-t D_1^2})\,dt$ this completes the proof of the first part of the index formula.
Now consider the problem on the cylinder type Lie manifold $(M_2, \A_2, \V_2)$. Here we apply the reduction given in Theorem \ref{Thm:cyl}.
This effectively reduces the problem to the calculation of $\indW(D_{\partial})$. The operator $D_{\partial}$ can be viewed
as an odd graded geometric Dirac operator on the Lie manifold $(Y, \B, \W)$. Hence we can use the above rescaling approch for the integrating groupoid $\H \rightrightarrows Y$.
With the same analysis as above, applied to the groupoid $\H$, we obtain the index formula also on the cylinder part of the decomposed Lie manifold. 
\end{proof}

% Fredholm conditions
If we impose additional conditions on the integrating groupoid, we can show that the renormalized index of Fredholm operators equals the Fredholm index.
In such a case the previous index theorem yields an actual generalization of Atiyah-Singer index theory.
For a given geometric admissible Dirac operator $D = D^W$ on a Lie manifold $(M, \A, \V)$ we denote by $\R(D) := \oplus_F \R_F(D)$ where the direct sum
ranges over all codimension one singular hyperfaces of $M$ and $\R_F(D)$ denotes the restriction of $D$ to the stratum $F$.
The \emph{indicial symbol} is more easily understood in the context of groupoids. If $\Dirac$ is the corresponding Dirac operator on an integrating Lie groupoid $\G \rightrightarrows M$, then
$\R_F(\Dirac) := (\Dirac_x)_{x \in F}$. The groupoid $\G \rightrightarrows M$ is called \emph{strongly amenable} if the natural 
action $C_r^{\ast}(\G) \hookrightarrow \L(\H)$ on the Hilbert space $\H := L^2(M_0)$ is injective.
\begin{Thm}
Let $(M, \A, \V)$ be a renormalizable Lie manifold for which there is an integrating Lie groupoid $\G \rightrightarrows M$ such that $\G$ is of polynomial growth,
Hausdorff and strongly amenable. Then for any geometric admissible Dirac operator $D = D^W$ with pointwise invertible indicial symbol $\R(D)$ we have
$\indV(D) = \ind(D)$.
\label{Thm:Fh}
\end{Thm}

\begin{proof}
From \cite{N} we have that $D$ is Fredholm if and only if $\R(D)$ is pointwise invertible by the conditions imposed on the groupoid $\G$.
Assume that $D$ is Fredholm. Then $\indV(D) = \ind(D)$ follows by the argument given in \cite[Section 2.2]{LMP}, see also \cite[Theorem 1.2]{BS}. 
\end{proof}

\begin{Exa}
The conditions on the Lie groupoid as stated in the Theorem hold in numerous special cases of Lie manifolds. We refer to \cite{N} for an overview. 
\label{Exa:Fh}
\end{Exa}

\section*{Acknowledgements}

I am very grateful to Jean-Marie Lescure for his insightful remarks concerning Section \ref{III} of this work during my visit to Clermont-Ferrand in December 2016.
I also thank Bernd Ammann, Magnus Goffeng, Victor Nistor and Elmar Schrohe for useful discussions.

\end{document}